\documentclass{amsart}

\usepackage{amssymb,amsthm,amsmath,amstext,amsfonts}

\usepackage[a4paper,top=2cm,bottom=2cm,left=2cm,right=2cm,marginparwidth=2cm]{geometry}

\usepackage{hyperref}
\hypersetup{colorlinks=true, allcolors=blue}

\usepackage{tikz-cd}

\newtheorem{Theorem}{Theorem}[section]
\newtheorem{thm}[Theorem]{Theorem}
\newtheorem{corl}[Theorem]{Corollary}
\newtheorem{lm}[Theorem]{Lemma}
\newtheorem{defn}[Theorem]{Definition}
\newtheorem{prop}[Theorem]{Proposition}
\newtheorem{rmk}[Theorem]{Remark}

\title{A restriction problem for mod-$p$ representations of $\mathrm{SL}_2(F)$}
\date{}
\author{Arpan Das}
\address{Harish-Chandra Research Institute, A CI of Homi Bhabha National
Institute, Chhatnag Road, Jhusi, Prayagraj - 211019, India}
\email{arpan3141@gmail.com}
\thanks{The author is supported by Ph.D. scholarship from Harish-Chandra Research Institute, A CI of Homi Bhabha National
Institute}

\begin{document}

\begin{abstract}
    Let $p$ be a prime and $F$ a non-archimedean local field of residue characteristic $p$. In this paper, we study the restriction of smooth irreducible $\bar{\mathbb{F}}_p$-representations of $\mathrm{SL}_2(F)$ to its Borel subgroup. In essence, we show that the action of $\mathrm{SL}_2(F)$ on its irreducibles is controlled by the action of the Borel subgroup. The results of this paper constitute the $\mathrm{SL}_2$-analogue of a work of Paškūnas\cite{PaskunasRestriction}.
\end{abstract}

\maketitle

\section{Introduction}
\subsection{Background}
Let $p$ be a prime number and $F$ a non-archimedean local field of residue characteristic $p$. We also fix an algebraically closed field $\bar{\mathbb{F}}_p$ of characteristic $p$. The subject of mod $p$ representations of $p$-adic groups began with the seminal work of Barthel and Livne \cite{Barthel} in 1994. A key classification result they proved constitutes subdividing the smooth irreducible $\bar{\mathbb{F}}_p$-representations (with central character) of $\mathrm{GL}_2(F)$ into four classes. These are called characters, principal series, Steinberg representations (a.k.a. special series), and supersingular representations. Although the first three types of representation are studied well beyond $\mathrm{GL}_2$, the last type, i.e. supersingulars, are rather mysterious. But when $F=\mathbb{Q}_p$, Breuil\cite{BreuilGL2Qp1} gave a complete classification of the supersingulars of $\mathrm{GL}_2(\mathbb{Q}_p)$ with explicit models for these representations. A similar classification of supersingulars for $\mathrm{SL}_2(\mathbb{Q}_p)$ was also obtained in \cite{Abdellatif} and \cite{Cheng}.

In 2007, Paškūnas proved the following theorem which showed that in a certain sense the action of $\mathrm{GL_2}(F)$ on its mod $p$ irreducible representations is controlled by the action of its Borel subgroup $B$. An interesting feature of this result of Paškūnas is that although it is a significant non-trivial result about the mysterious supersingulars, the proof does not use anything other than the very definition of supersingulars and some natural technical results about the way the supersingulars are parameterized.

\begin{thm}[Theorem 1.1 in \cite{PaskunasRestriction}]\label{Paskunas' main result}
    Let $\pi$ and $\pi^{\prime}$ be smooth representations of $\mathrm{GL}_2(F)$ over $\bar{\mathbb{F}}_p$, and $B$ denote the Borel subgroup of upper triangular matrices in $\mathrm{GL}_2(F)$. Suppose $\pi$ is irreducible with a central character. Then the following hold : 
    \begin{enumerate}
        \item If $\pi$ is a principal series representation then $\pi|_{B}$ is of length $2$; otherwise $\pi|_{B}$ is irreducible.

        \item If $\pi\neq \mathrm{St},$ the Steinberg representation, then $\mathrm{Hom}_{B}(\pi,\pi^{\prime})\cong \mathrm{Hom}_{\mathrm{GL_2}(F)}(\pi,\pi^{\prime})$; otherwise we have $\mathrm{Hom}_{B}(\mathrm{St},\pi^{\prime})\cong \mathrm{Hom}_{\mathrm{GL_2}(F)}(\mathrm{Ind}_{B}^{\mathrm{GL_2}(F)}(1),\pi^{\prime})$. 
    \end{enumerate}
\end{thm}

When $F=\mathbb{Q}_p$ and $\pi^{\prime}$ is irreducible as well, Berger \cite{Berger} proved the above theorem using the arithmetic of $(\varphi,\Gamma)$-modules, and also the explicit models for the supersingulars of $\mathrm{GL}_2(\mathbb{Q}_p)$ (see \cite{BreuilGL2Qp1}). The restriction of mod $p$ representations to Borel subgroup plays a crucial role in the work of Colmez \cite{Colmez} on $p$-adic representations of $\mathrm{GL}_2(\mathbb{Q}_p)$. The theorem of Paškūnas generalized Berger's result by using only representation-theoretic methods.

\subsection{Our results}

In the present paper, we consider the $\mathrm{SL}_2$-analogue of Theorem \ref{Paskunas' main result}. In particular, our main result for non-supersingular representations is the following theorem.
\begin{thm}[Theorem \ref{lifting B_S intertwiners from V_eta to G_S intertwiners from Ind_BS^GS}]\label{main theorem 2}
    Let $G_S:=\mathrm{SL_2}(F)$, and $B_S$ its Borel subgroup. Let $\eta$ be a smooth $\bar{\mathbb{F}}_p$-character of $B_S$.  Given a smooth $\bar{\mathbb{F}}_p$-representation $\pi$ of $G_S$, the restriction map induces an isomorphism between the following spaces of intertwiners : $$\mathrm{Hom}_{G_S}(\mathrm{Ind}_{B_S}^{G_S}(\eta),\pi)\cong \mathrm{Hom}_{B_S}(V_{\eta},\pi|_{B_S}).$$ Here, $V_{\eta}$ is the kernel of the map $\mathrm{Ind}_{B_S}^{G_S}(\eta)\to \eta$ that evaluates every function at the identity matrix.
\end{thm}

Then, as a consequence of the above result, we prove the following theorem, which is similar to part (2) of Theorem \ref{Paskunas' main result}.

\begin{thm}[Corollary \ref{intertwiners from principal series}, Corollary \ref{intertwiners from steinberg}]\label{main thm 3}
    Let $\pi$ be a smooth representation of $G_S$ over $\bar{\mathbb{F}}_p$. If $\eta\neq 1$ is a smooth $\bar{\mathbb{F}}_p$-character of $B_S$, then $$\mathrm{Hom}_{B_S}(\mathrm{Ind}_{B_S}^{G_S}(\eta),\pi^{\prime})\cong \mathrm{Hom}_{G_S}(\mathrm{Ind}_{B_S}^{G_S}(\eta),\pi^{\prime});$$ Otherwise we have the following isomorphism $$\mathrm{Hom}_{B_S}(\mathrm{St}_S|_{B_S},\pi)\cong \mathrm{Hom}_{G_S}(\mathrm{Ind}_{B_S}^{G_S}(1),\pi),$$ where $\mathrm{St}_S:=\frac{\mathrm{Ind}_{B_S}^{G_S}(1)}{1}$ is the mod $p$ Steinberg representation. The latter isomorphism cannot be improved by replacing $\mathrm{St}_S$ with $\mathrm{Ind}_{B_S}^{G_S}(1)$.
\end{thm}

We let $K_0$ and $K_1$ denote the maximal compact subgroups of $\mathrm{SL}_2(F)$. Then, our main result for the supersingular representations of $\mathrm{SL}_2(F)$ is the following. 

\begin{thm}[Theorem \ref{restriction of admissible supersingulars to borel is irreducible}, Theorem \ref{G_S maps are B_S maps from admissible supersingulars}]\label{main thm 4}
    Let $K\in \{K_0,K_1\}$, let $\pi$ be a $K$-supersingular representation, and $\pi^{\prime}$ a smooth representation of $G_S$ over $\bar{\mathbb{F}}_p$. Then, we have : 
    \begin{enumerate}
        \item $\pi|_{B_S}$ is an irreducible $B_S$-representation.

        \item The restriction map gives an isomorphism : $\mathrm{Hom}_{G_S}(\pi,\pi^{\prime})\cong \mathrm{Hom}_{B_S}(\pi,\pi^{\prime})$.
    \end{enumerate}
\end{thm}
A small subtlety must be pointed out here. We have stated our main result by breaking it into two parts, for non-supersingulars and supersingulars. This is because Theorem \ref{Paskunas' main result} implicitly assumes that smooth irreducibles of $\mathrm{GL}_2(F)$ (admitting central characters) are classified into four types; however, such a classification for $\mathrm{SL}_2(F)$ is available only with the condition of admissibility (see Théorème 0.6 in \cite{Abdellatif}).

We mention that in our initial attempt to prove the result for supersingulars of $\mathrm{SL}_2(F)$, we considered only admissible supersingulars. Such a result would have been slightly less valuable owing to some recent results on the existence of smooth irreducible non-admissible $\bar{\mathbb{F}}_p$-representations (for example \cite{Le1} and \cite{Le2}). It was suggested to the author by Peng Xu that this admissibility condition can be removed by proving an analogue of Proposition 32 of \cite{Barthel}. The author thanks Peng Xu for this suggestion; in fact, the methods we have used to prove the following Proposition are very close to the recent work of Xu \cite{XuHeckeEigenvalues} on the related rank 1 quasi-split semisimple group $U(2,1)$. It was pointed out in the introduction of Xu's paper \cite{XuHeckeEigenvalues} that other than $U(2,1)$, the only other example for which an analogue of this theorem is known is $\mathrm{GL}_2(F)$ (more precisely, smooth irreducibles of $\mathrm{GL}_2(F)$ having central characters; see Proposition 32 of \cite{Barthel}). We have therefore proved the following result which may be of independent interest (see Question 8 in \cite{abe2017questionsmodprepresentations}), providing a third example $\mathrm{SL}_2$ for which Hecke eigenvalues exist:
\begin{prop}[Proposition \ref{finite dimensionality result}]\label{Hecke eigenvalue 0}
    Let $\sigma$ be a smooth irreducible $\bar{\mathbb{F}}_p$-representation of $K_0$, and let $\pi$ be a smooth irreducible $\bar{\mathbb{F}}_p$-representation of $G_S$. If $\varphi\in \mathrm{Hom}_{G_S}(\mathrm{ind}_{K_0}^{G_S}(\sigma),\pi)$ is non-zero, the $\mathrm{End}_{\bar{\mathbb{F}}_p[G_S]}(\mathrm{ind}_{K_0}^{G_S}(\sigma))$-right-submodule of $\mathrm{Hom}_{G_S}(\mathrm{ind}_{K_0}^{G_S}(\sigma),\pi)$ generated by $\varphi$ is of finite dimension.
\end{prop}

Finally, we mention that there is some consensus that the restriction theorem of Paškūnas may be true for other semisimple $p$-adic groups of rank 1. This was announced as a forthcoming joint work by Abdellatif and Hauseux \cite{Abdellatif-Hauseux}. However, they have considered only admissible supersingular representations. In the present paper, we focus on the split semisimple rank 1 group $\mathrm{SL}_2(F)$, since we are able make use of several finer structural results (for example Propositions \ref{action of degenerate IHA on isoty-comp of cpt ind}, \ref{action of non-degenerate IHA on isoty-comp of cpt ind}, and \ref{finite codimension}) to push the main result beyond admissible representations. The analogues of some of these crucial results (for example Proposition \ref{finite codimension}) may not be true for general classes of $p$-adic groups (see Theorem 5.1 in \cite{XuHeckeEigenvalues}). 
  
\subsection{Organization of the paper}
Our paper is organised as follows : In Section \ref{background}, we set up the notations, and state some standard facts about smooth representations of locally profinite groups. In Section \ref{weights and carter lusztig}, we recall the main theorem of Carter-Lusztig theory specialized to $\mathrm{SL}_2(\mathbb{F}_q)$. In Section \ref{Hecke algebras and eigenvalues}, we first recall some standard results about the structure of the spherical Hecke algebra. Then we prove certain structural results related to the action of Iwahori Hecke algebras on the isotypic components of the compact inductions of weights. Using these structural results, we then prove certain important finite-dimensionality results (Proposition \ref{finite codimension} and Proposition \ref{finite dimensionality result}). In Section \ref{non-supsing case}, we prove Theorem \ref{main theorem 2} and Theorem \ref{main thm 3} for non-supersingulars. Finally, in Section \ref{supsing case}, we prove Theorem \ref{main thm 4} for supersingulars.

\section{Generalities and notations}\label{background}

 We recall some standard facts from the theory of mod $p$ representations of locally profinite groups in Section \ref{Generalities}. For more details, the reader can see Section 2 of \cite{Barthel}.

\subsection{Compact induction and Hecke algebras}\label{Generalities} We take $p$ to be a prime throughout, and $\Bar{\mathbb{F}}_p$ a fixed algebraic closure of the finite field $\mathbb{F}_p$ with $p$ elements. All representations in this article, unless otherwise mentioned, are considered over $\bar{\mathbb{F}}_p$. We recall some general results. Just for this subsection, we let $G$ be any locally profinite group, and $H$ some closed subgroup. Recall that a representation $\pi$ of $G$ is called \textit{smooth} if every vector $v\in \pi$ is fixed by some compact open subgroup of $G$. A representation $\pi$ of $G$ is called \textit{admissible} if for every compact open subgroup $K$ of $G$, the space of invariants $\pi^K$ is finite dimensional. 

Let $\sigma$ be a smooth $\bar{\mathbb{F}}_p$-representation of $H$. We consider the following space of functions : $$\text{IND}_H^G(\sigma):=\{f:G\to \sigma\,|\, f(hg)=\sigma(h)(f(g)),\,\forall g\in G,h\in H\}.$$
Then, $G$ acts on $\text{IND}_H^G(\sigma)$ via $(g\cdot f)(g^{\prime}):=f(g^{\prime}g)$. The smooth part of $\text{IND}_H^G(\sigma)$, that is, vectors that have open stabilizers, is denoted by $\text{Ind}_H^G(\sigma)$, and this subrepresentation is called the \textit{smooth induction} of $\sigma$. The subrepresentation of $\text{Ind}_H^G(\sigma)$ consisting of functions $f$ such that the image of its support $\text{Supp}(f)$ inside $H\backslash G$ is compact (equivalently, finite, whenever $H$ is also open) is denoted by $\text{c-Ind}_H^G(\sigma)$ or $\text{ind}_H^G(\sigma)$, and is called the \textit{compact induction} of $\sigma$.

In practice, whenever we use compact induction the subgroup $H$ is typically considered to be open as well. So, for the remaining part of this subsection we take $H$ to be an open subgroup of $G$. Then, by virtue of the $H$-linearity, the support of any $f\in \text{ind}_H^G(\sigma)$ can be written as a finite disjoint union of right $H$-cosets. We define some standard functions in $\text{ind}_H^G(\sigma)$. For $g\in G$ and $v\in \sigma$ we define : $$[g,v](x):=\begin{cases}
    \sigma(xg)(v) & \text{if}\,x\in Hg^{-1}\\
    0 & \text{otherwise}
\end{cases}.$$ 
It can be checked that $g\cdot [g^{\prime},v]=[gg^{\prime},v]$ and $[gh,v]=[g,\sigma(h)(v)]$ for every $g,g^{\prime}\in G$ and $h\in H$. Also, any $f\in \text{ind}_H^G(\sigma)$ can be written as $$f=\sum\limits_{Hg\in \text{Supp($f$)}}[g^{-1},f(g)].$$

We define $\mathbb{H}_{\bar{\mathbb{F}}_p}(G,H,\sigma)$ to be the $\bar{\mathbb{F}}_p$-algebra of functions $\Phi:G\to \text{End}_{\bar{\mathbb{F}}_p}(\sigma)$ which satisfy the following conditions:
\begin{enumerate}
    \item $\Phi(hgh^{\prime})=\sigma(h)\circ \Phi(g)\circ \sigma(h^{\prime})$ for all $h,h^{\prime}\in H$ and $g\in G$.

    \item For each $v\in \sigma$, the map $g\mapsto \Phi(g)(v):G\to \sigma$ is locally constant with the image of its support in $H\backslash G$ being compact.
\end{enumerate}
We equip $\mathbb{H}_{\bar{\mathbb{F}}_p}(G,H,\sigma)$ with the convolution product : $(\Phi_1\star \Phi_2)(g):=\sum_x\Phi_1(x)\Phi_2(x^{-1}g),$ where $x$ varies over a system of representatives of $G/H$ in $G$. We can check that this sum is independent of the choice of representatives, and that the sum is finite when evaluated on some vector in $\sigma$.

On the other hand, we have the $\bar{\mathbb{F}}_p$-algebra $\mathcal{H}_{\bar{\mathbb{F}}_p}(G,H,\sigma):=\text{End}_{\bar{\mathbb{F}}_p[G]}(\text{ind}_H^G(\sigma))$ of $G$-intertwiners. Then, the map $$\eta : \mathbb{H}_{\bar{\mathbb{F}}_p}(G,H,\sigma)\to \mathcal{H}_{\bar{\mathbb{F}}_p}(G,H,\sigma),$$ given by : $$\eta(\Phi)(f)(g):=\sum\limits_{x\in G/H}\Phi(x)f(x^{-1}g),$$ for $\Phi\in \mathbb{H}_{\bar{\mathbb{F}}_p}(G,H,\sigma),\, f\in \text{ind}_H^G(\sigma),$ and $g\in G$, is an isomorphism of algebras with inverse given by : $$\eta^{-1}(T)(g)(v)=T([1,v])(g),$$ for $T\in \mathcal{H}_{\bar{\mathbb{F}}_p}(G,H,\sigma),\, g\in G,\, v\in \sigma$. 

In the present article we will be dealing with the situation when $\sigma$ is finite dimensional. For this we make some observations. At first, note that in this case $\mathbb{H}_{\bar{\mathbb{F}}_p}(G,H,\sigma)$ consists of functions $\Phi:G\to \text{End}_{\bar{\mathbb{F}}_p}(\sigma)$ which are locally constant with the image of the support in $H\backslash G$ compact, and satisfying the condition (1) above. Also, the support of $\Phi$ can be written as a finite disjoint union of double cosets in $H\backslash G/H$. If $T_{\Phi}\in \mathcal{H}_{\bar{\mathbb{F}}_p}(G,H,\sigma)$ denotes the endomorphism associated to $\Phi$ by the above isomorphism $\eta$, then it is easy to check that $$T_{\Phi}([g,v])=\sum\limits_{yH\in G/H}[gy,\Phi(y^{-1})(v)],$$ for any standard function $[g,v]\in \text{ind}_H^G(\sigma)$. Finally, when $\Phi$ is supported only on one double coset, say $Hg_0H$, we can write $Hg_0H$ as a finite union of right $H$-cosets because of the support condition on $\Phi$. So, we can write $Hg_0^{-1}H=\bigsqcup_{i=1}^mk_ig_0^{-1}H$, and replacing $yH$ by $k_ig_0^{-1}H$ we have: 
\begin{equation}\label{action of a general hecke operator on a standard function}
T_{\Phi}([g,v])=\sum\limits_{i=1}^m[gk_ig_0^{-1},\Phi(g_0)\sigma(k_i^{-1})v].    
\end{equation}

Finally, we mention the \textit{Frobenius reciprocity} for compact induction : Let $G$ be a locally profinite group, and $H$ an open subgroup $G$. Let $\pi$ be a smooth representation of $G$, and $\sigma$ a smooth representation of $H$. Then, the map $$\psi\mapsto [w\mapsto \psi([1,w]): \mathrm{Hom}_G(\mathrm{ind}_H^G(\sigma),\pi)\to \mathrm{Hom}_H(\sigma,\pi|_H)$$ is an isomorphism of vector spaces. The inverse is given by $$\varphi\mapsto [f\mapsto \sum_{g\in H\backslash G}\pi(g^{-1})\varphi(f(g))]: \mathrm{Hom}_H(\sigma,\pi|_H)\to \mathrm{Hom}_G(\mathrm{ind}_H^G(\sigma),\pi).$$
The proofs of all the facts mentioned in this subsection are fairly routine and can be found in \cite{Barthel}.

\subsection{Notations}\label{notations}

 Let $F$ be a non-archimedean local field of residual characteristic $p$,  with $\mathcal{O}_F$ its ring of integers, $\mathfrak{p}_F$ its maximal ideal, $\omega_F$ a fixed uniformizer, and $k_F:=\mathcal{O}_F/\mathfrak{p}_F$ the residue field of cardinality say $q=p^n$. Borrowing the notations of \cite{Abdellatif}, we set $G_S:=\mathrm{SL_2}(F)$. The corresponding standard maximal compact subgroups are $K_0:=\mathrm{SL_2}(\mathcal{O}_F)$, and $K_1:=\begin{pmatrix}
    1 & 0\\
    0 & \omega_F
\end{pmatrix}K_0\begin{pmatrix}
    1 & 0\\
    0 & \omega_F^{-1}
\end{pmatrix}$. Let $I_S$ denote the Iwahori subgroup, and $I_S(1)$ the pro-$p$-Iwahori subgroup of $G_S$. We have $$I_S=\begin{pmatrix}
    \mathcal{O}_F^{\times} & \mathcal{O}_F\\
    \mathfrak{p}_F & \mathcal{O}_F^{\times} 
\end{pmatrix}\cap K_0,\: I_S(1)=\begin{pmatrix}
    1+\mathfrak{p}_F & \mathcal{O}_F\\
    \mathfrak{p}_F &  1+\mathfrak{p}_F
\end{pmatrix}\cap K_0.$$ We denote by $B_S$ the Borel subgroup (of upper triangular matrices) in $G_S$, and by $T_S$ we denote the diagonal matrices in $B_S$. We also denote by $U_S(\mathfrak{p}_F^n)$ (resp. $\bar{U}_S(\mathfrak{p}_F^n)$) the upper unipotent (resp. lower unipotent) matrices with the top right (resp. bottom left) entry in $\mathfrak{p}_F^n$, for $n\in \mathbb{Z}$. We set $$\alpha_0:=\begin{pmatrix}
    \omega_F^{-1} & 0\\
    0 & \omega_F
\end{pmatrix}, w_0:=\begin{pmatrix}
    0 & -1\\
    1 & 0
\end{pmatrix}, \beta_0:=\alpha_0w_0=\begin{pmatrix}
    0 & -\omega_F^{-1}\\
    \omega_F & 0
\end{pmatrix}.$$
Finally, for a tuple $\lambda=(\lambda_0,...,\lambda_{m-1})\in k_F^m$, we set $$A(\lambda):=\sum_{i=0}^{m-1}[\lambda_i]\omega_F^i\in \mathcal{O}_F$$ where $[\cdot]:\mathbb{F}_q^{\times}\to \mathcal{O}_F^{\times}$ denotes the \textit{multiplicative lift}, and we set $[0]:=0$.

\section{Carter-Lusztig theory for $\mathrm{SL}_2(\mathbb{F}_q)$}\label{weights and carter lusztig}

For this section we take $\Gamma$ to be $\mathrm{SL_2}(\mathbb{F}_q)$, $B$ as the subgroup of upper triangular matrices, $U$ as the subgroup of upper unipotent matrices, and $T$ as the diagonal matrices in $B$. In this section we recall the theory of mod $p$ representations of $\mathrm{SL}_2(\mathbb{F}_q)$. We do this using the very elegant theory of Carter and Lusztig which gives a uniform construction of all mod $p$ irreducibles of finite groups with a split $BN$-pair. For our purpose we specialize this theory to the group $\Gamma$ as that is what we need in this paper. We will present the main results without proof. The interested reader can see the beautiful paper of Carter and Lusztig \cite{Carter-Lusztig}, where the proofs are of fairly elementary nature.

Given a character $\chi:T\to \bar{\mathbb{F}}_p^{\times}$, we can consider it to be a character, denoted again by $\chi$, of $B$ by setting $\chi|_{U}=1$. Then, we define a function $\varphi_{\chi}\in \text{Ind}_B^{\Gamma}(\chi)$ such that $\varphi|_{Bw_0U}=0$ and $\varphi(I_2)=1$, where $I_2$ denotes the identity matrix. We have that $\varphi_{\chi}$ generates $\text{Ind}_B^{\Gamma}(\chi)$ as $\bar{\mathbb{F}}_p[\Gamma]$-module. Next, we define an $\bar{\mathbb{F}}_p[{\Gamma}]$-module endomorphism $T_{w_0}$ of $\text{Ind}_U^{\Gamma}(1)=\{f:U\backslash \Gamma\to \bar{\mathbb{F}}_p\}$  by $$T_{w_0}(f)(Ug):=\sum_{Ug^{\prime}\subset Uw_0^{-1}Ug}f(Ug^{\prime})\: \text{ for }f\in \text{Ind}_U^{\Gamma}(1).$$ Then, $T_{w_0}$ restricts to a map $T_{w_0}: \text{Ind}_B^{\Gamma}(\chi)\to \text{Ind}_B^{\Gamma}(\chi^{w_0}).$ Also $\text{Ind}_B^{\Gamma}(\chi)^{U}$ is of dimension $2$, and generated by the functions $\varphi_{\chi}$ and $T_{w_0}\varphi_{\chi^{w_0}}$. Note that $T_{w_0}:\text{Ind}_B^{\Gamma}(\chi^{w_0})\to \text{Ind}_B^{\Gamma}(\chi)$. So, $T_{w_0}^2:\text{Ind}_B^{\Gamma}(\chi)\to \text{Ind}_B^{\Gamma}(\chi)$ for any character $\chi$ of $B$. It can be shown (see Proposition 3.15 of \cite{Carter-Lusztig}) that
\begin{equation}\label{action of T_w_0 on generator}
    T_{w_0}\varphi_{\chi}=\sum_{\lambda\in \mathbb{F}_q}\begin{pmatrix}1 & \lambda\\0 & 1\end{pmatrix}w_0^{-1}\cdot \varphi_{\chi^{w_0}}.
\end{equation}
We also have : $T_{w_0}^2=\begin{cases} 
      0 ,& \chi\neq 1 \\
      -T_{w_0}, & \chi=1 
   \end{cases}.$
Next, for a character $\chi$ of $B$, we set $$J_0(\chi):=\begin{cases} 
      \emptyset, & \chi\neq 1 \\
      \{1\}, & \chi=1 
   \end{cases}.$$ Then, for each subset $J\subset J_0(\chi)$, we define an intertwiner $\Theta_{w_0}^{J}:\text{Ind}_B^{\Gamma}(\chi)\to \text{Ind}_B^{\Gamma}(\chi^{w_0})$ as follows : $$\Theta_{w_0}^{J}:=\begin{cases} 
      T_{w_0}, & J=\emptyset \\
      Id+T_{w_0}, & J=\{1\}. 
   \end{cases}$$
We set $f_{\chi}^{J}=\Theta_{w_0}^{J}\varphi_{\chi}$. 
Now we can state the main theorem of Carter-Lusztig theory.

\begin{thm}\label{Carter-Lusztig thory}[Corollary 6.5, Theorem 7.1, Corollary 7.2, and Theorem 7.4 in \cite{Carter-Lusztig}]
For each pair $(\chi,J\subset J_0(\chi))$, the only $U-$invariant vectors in $\Theta_{w_0}^{J}(\mathrm{Ind}_B^{\Gamma}(\chi))$ are scalar multiples of $f_{\chi}^{J}$. The module $\Theta_{w_0}^{J}(\mathrm{Ind}_B^{\Gamma}(\chi))$ is an irreducible submodule of $\mathrm{Ind}_B^{\Gamma}(\chi^{w_0})$ generated by $f_{\chi}^J$, and the subgroup $B$ acts on the line $\bar{\mathbb{F}}_p\cdot f_{\chi}^J$ via the character $\chi$. For distinct pairs $(\chi,J)$ the corresponding modules $\Theta_{w_0}^{J}(\mathrm{Ind}_B^{\Gamma}(\chi))$ are non-isomorphic. The modules $\Theta_{w_0}^{J}(\mathrm{Ind}_B^{\Gamma}(\chi))$ for $J\subset J_0(\chi)$ are the only irreducible submodules of $\mathrm{Ind}_B^{\Gamma}(\chi^{w_0})$. Every irreducible $\bar{\mathbb{F}}_p[{\Gamma}]$-module is isomorphic to $\Theta_{w_0}^{J}(\mathrm{Ind}_B^{\Gamma}(\chi))$ for some pair $(\chi,J)$.
\end{thm}

\begin{rmk}\label{steinberg mod p}
\begin{enumerate}
    \item Note that if $\sigma$ is a mod $p$ irreducible representation of $\mathrm{SL_2}(\mathcal{O}_F)$, then every vector in $\sigma$ is fixed by $K_0(1):=\begin{pmatrix}1+\mathfrak{p}_F & \mathfrak{p}_F\\ \mathfrak{p}_F & 1+\mathfrak{p}_F\end{pmatrix}\cap K_0$, as $K_0(1)$ is a pro-$p$-group and normal in $K_0$. Hence, we can consider $\sigma$ to be a representation of $K_0/K_0(1)\simeq \mathrm{SL_2}(k_F)$(via the mod $p$ reduction map). Therefore, mod $p$ irreducibles of $K_0$ and $\mathrm{SL_2}(\mathbb{F}_q)$ are essentially same, and they are called \textbf{Serre weights} or simply \textbf{weights} of $\mathrm{SL_2}(\mathbb{F}_q)$ or $\mathrm{SL_2}(\mathcal{O}_F)$. The irreducibles of $K_1$ are then the conjugate representations $\sigma^{\alpha}$, where $\alpha={\begin{pmatrix}
    1 & 0\\
    0 & \omega_F
    \end{pmatrix}}$. Hence, we can think of the generating vectors $f_{\chi}^J$ as vectors fixed by $I_S(1)$ on which $I_S$ acts by the character $\chi\circ \mathrm{red_p}$. Also, note that $f^{J}_{\chi}$ is an eigenvector of the operator $T_{w_0}$.

    \item The traditional way in which mod $p$ irreducibles of $\mathrm{SL_2}(\mathbb{F}_q)$ are realized is by using the symmetric powers. Let $r\in \{0,\dots,q-1\}$. We write $r=r_0+r_1p+\cdots+r_{n-1}p^{n-1}$, and denote $(r_0,\dots,r_{n-1})$ by $\Vec{r}$. Let $\sigma_{r_l}:=\text{Sym}^{r_l}\bar{\mathbb{F}}_p^2=\bigoplus\limits_{i=0}^r\bar{\mathbb{F}}_p X^{r_l-i}Y^i$ denote the representation of $\mathrm{SL_2}(\mathbb{F}_q)$ on which the action of a matrix $\begin{pmatrix}
    a & b\\
    c & d
    \end{pmatrix}\in \mathrm{SL_2}(\mathbb{F}_q)$ is given by $$\begin{pmatrix}
    a & b\\
    c & d
    \end{pmatrix}\cdot (X^{r_l-i}Y^i):=(aX+cY)^{r_l-i}(cX+dY)^i.$$ Let $\sigma_{\Vec{r}}:=\sigma_{r_0}\otimes \sigma_{r_1}\otimes\cdots\otimes \sigma_{r_{n-1}}$, and on this space we define the action of $\begin{pmatrix}
    a & b\\
    c & d
    \end{pmatrix}\in \mathrm{SL_2}(\mathbb{F}_q)$ by $$\begin{pmatrix}
    a & b\\
    c & d
    \end{pmatrix}\cdot (\bigotimes_{l=0}^{n-1}X^{r_l-i_l}Y^{i_l}):=\bigotimes_{l=0}^{n-1}(\begin{pmatrix}
    a^{p^l} & b^{p^l}\\
    c^{p^l} & d^{p^l}
    \end{pmatrix}\cdot X^{r_l-i_l}Y^{i_l}) $$ We write $\sigma_{\Vec{r}}$ as $\mathrm{Sym}^{r_0}\bar{\mathbb{F}}_p^2\otimes (\mathrm{Sym}^{r_1}\bar{\mathbb{F}}_p^2)^{\mathrm{Fr}}\otimes\cdots \otimes (\mathrm{Sym}^{r_{n-1}}\bar{\mathbb{F}}_p^2)^{\mathrm{Fr}^{n-1}},$ where the notation $\mathrm{Fr}^i$ indicates that the action on the $i$-th component is twisted by the $i$-th power of the standard Frobenius.

    \item For $\chi\neq 1$ clearly $\Theta_{w_0}^{\emptyset}(\mathrm{Ind}_B^{\Gamma}(\chi))$ are non-trivial weights of $K_0$. In fact, one can show that if $\chi$ is the $r$-th power map on $\mathbb{F}_q^{\times}$ for some $0\leq r\leq q-1$, then writing $r=r_0+r_1p+\cdots+r_{n-1}p^{n-1}$, we have $$\Theta_{w_0}^{\emptyset}(\mathrm{Ind}_B^{\Gamma}(\chi))\cong \mathrm{Sym}^{r_0}\bar{\mathbb{F}}_p^2\otimes (\mathrm{Sym}^{r_1}\bar{\mathbb{F}}_p^2)^{\mathrm{Fr}}\otimes\cdots \otimes (\mathrm{Sym}^{r_{n-1}}\bar{\mathbb{F}}_p^2)^{\mathrm{Fr}^{n-1}}.$$ For $\chi=1$, we can show that the irreducible representation $\Theta_{w_0}^{\emptyset}(\mathrm{Ind}_B^{\Gamma}(1)$ is non-trivial. In fact, in this case we have $$\Theta_{w_0}^{\emptyset}(\mathrm{Ind}_B^{\Gamma}(1))\cong \mathrm{Sym}^{p-1}\bar{\mathbb{F}}_p^2\otimes (\mathrm{Sym}^{p-1}\bar{\mathbb{F}}_p^2)^{\mathrm{Fr}}\otimes\cdots \otimes (\mathrm{Sym}^{p-1}\bar{\mathbb{F}}_p^2)^{\mathrm{Fr}^{n-1}}.$$ This is proved for the group $\mathrm{GL_2}(\mathbb{F}_q)$ in Proposition $3.2.2$ of Paškūnas' book \cite{PaskunasCoeff}. The corresponding proof for $\mathrm{SL_2}(\mathbb{F}_q)$ can be reproduced line-by-line with obvious changes. The representation $\Theta_{w_0}^{\emptyset}(\mathrm{Ind}_B^{\Gamma}(1))$ is the so called mod $p$ Steinberg representation.

\end{enumerate}
\end{rmk}

\section{Hecke algebras and eigenvalues}\label{Hecke algebras and eigenvalues}
In this section we recall some familiar structural results about the spherical and Iwahori Hecke algebras, and prove a certain finiteness result analogous to Proposition 32 of \cite{Barthel}.

\subsection{Spherical Hecke Algebras}\label{spherical Hecke alg structure} In this article we will typically use the notation $\sigma_{\chi,J}$ or $\sigma_{\chi}$ (where $\chi$ is a character of $\mathbb{F}_q^{\times}$) or $\sigma_{\Vec{r}}$ or simply $\sigma$ to denote weights of $K_0$, whichever best suites the context. Given a weight $\sigma_{\chi,J}$ of $K_0$, it is of interest to know the structure of the \textit{spherical Hecke algebra} $\mathcal{H}_{\bar{\mathbb{F}}_p}(G_S,K_0,\sigma_{\chi,J})$. It turns out that we have $$\mathcal{H}_{\bar{\mathbb{F}}_p}(G_S,K_0,\sigma_{\chi,J})=\bar{\mathbb{F}}_p[\tau],$$ for a single Hecke operator $\tau\in \mathcal{H}_{\bar{\mathbb{F}}_p}(G_S,K_0,\sigma_{\chi,J})$. The action of this operator $\tau$ can be explicitly computed on the standard function $\varphi:=[1,f^J_{\chi}]$ that generates $\mathrm{ind}_{K_0}^{G_S}(\sigma_{\chi,J})$ as follows : \begin{equation}\label{action of tau for sigma non-trivial}
    \tau(\varphi)=\sum_{\lambda\in k_F^2}\begin{pmatrix}1 & A(\lambda)\\ 0 & 1\end{pmatrix}\alpha_0^{-1}\varphi,
\end{equation}whenever $\sigma_{\chi,J}\neq 1$, and \begin{equation}
    \tau(\varphi)=\sum_{\lambda\in k_F^2}\begin{pmatrix}1 & A(\lambda)\\ 0 & 1\end{pmatrix}\alpha_0^{-1}\varphi+\sum_{\mu\in k_F}\begin{pmatrix}1 & 0\\ \omega_FA(\mu) & 1\end{pmatrix}\alpha_0 \varphi,
\end{equation} whenever $\sigma_{\chi,J}=1$.

All the above results can be found in Section 3.2 of \cite{Abdellatif}. The explicit formula for the action of $\tau$ on $\varphi$ above is deduced from Corollaire 3.12 of \cite{Abdellatif}, which in turn is derived using the equation (\ref{action of a general hecke operator on a standard function}) of Subsection \ref{Generalities}. At last, we give the following definition motivated by the action of $\tau$.

\begin{defn}
    Let $\pi$ be a smooth representation of $G_S$. We define the map $$\mathcal{S}:=[v\mapsto \sum_{\lambda\in k_F^2}\begin{pmatrix}1 & A(\lambda)\\ 0 & 1\end{pmatrix}\alpha_0^{-1}\cdot v]:\pi\to \pi.$$
\end{defn}

\subsection{Iwahori-Hecke algebras} In this subsection we will recall some general structural results about the Iwahori-Hecke algebras. 

Let $\pi$ be a smooth representation of $G_S$. We know by Frobenius reciprocity that $$\mathrm{Hom}_{G_S}(\mathrm{ind}_{I_S(1)}^{G_S}(1),\pi)\simeq \pi^{I_S(1)}.$$ As a result, $\pi^{I_S(1)}$ carries a natural right action of $\mathcal{H}(G_S,I_S(1),1)$, the pro-$p$-Iwahori Hecke algebra. We have the following Proposition.

\begin{prop}\label{right action of pro p IHA}
    Let $\pi$ be a smooth representation of $G_S$, and $v\in \pi^{I_S(1)}$. Let $T_g\in \mathcal{H}(G_S,I_S(1),1)$ be the operator corresponding to the function $\Phi_g\in \mathbb{H}_{\bar{\mathbb{F}}_p}(G_S,I_S(1),1)$ which is supported on $I_S(1)gI_S(1)$ and such that $\Phi_g(g)=1$. Then, we have $$v\,|\,T_g=\sum_{i\in I_S(1)/(I_S(1)\cap (g^{-1}I_S(1)g))}ig^{-1}\cdot v.$$In particular, for $g\in \{w_0,w_0^{-1}\alpha_0^{-1}\}$, we have $$v\,|\,T_{w_0}=\sum_{\lambda\in k_F}\begin{pmatrix}
        1 & [\lambda]\\ 0 & 1
    \end{pmatrix}w_0^{-1}\cdot v\quad \text{and}\quad v\,|\,T_{w_0^{-1}\alpha_0^{-1}}=\sum_{\mu\in k_F}w_0\begin{pmatrix}
        1 & [\mu]\omega_F\\ 0 & 1
    \end{pmatrix}\alpha_0^{-1}\cdot v.$$ 
\end{prop}

\begin{proof}
    Using Proposition 6 of \cite{Barthel} (or by the explicit bijection between the convolution and endomorphism algebras, and Frobenius reciprocity, mentioned in subsection \ref{Generalities}), we have $$f\,|\,T_g=\sum_{g^{\prime}\in I_S(1)\backslash G_S}\Phi_g(g^{\prime})(g^{\prime})^{-1}\cdot f.$$Since the above sum is supported only on $I_S(1)gI_S(1)$, we decompose this double coset into $I_S(1)gi_1^{-1}\sqcup\cdots \sqcup I_S(1)gi_n^{-1}$. Note that $I_S(1)g i^{-1}=I_S(1)g j^{-1}$ if and only if $i^{-1}j\in g^{-1}I_S(1)g$, whence we can take $i_1,\dots,i_n$ as representatives of $I_S(1)/(I_S(1)\cap (g^{-1}I_S(1)g))$. These reductions allow us to write the above sum as $$f\,|\, T_g=\sum_{i\in I_S(1)/(I_S(1)\cap (g^{-1}I_S(1)g))}i g^{-1}\cdot f.$$ The formulas for $g=w_0$ and $g=w_0^{-1}\alpha_0^{-1}$ are obtained by noting that the Iwahori decomposition $I_S(1)=U_S(\mathcal{O}_F)\times T_S(1+\mathfrak{p}_F)\times \bar{U}_S(\mathfrak{p}_F)$ gives the following isomorphisms : $$I_S(1)/(I_S(1)\cap (w_0^{-1}I_S(1)w_0))\simeq U_S(\mathcal{O}_F)/U_S(\mathfrak{p}_F)\quad \text{and}\quad I_S(1)/(I_S(1)\cap (\alpha_0w_0I_S(1)w_0^{-1}\alpha_0^{-1}))\simeq \bar{U}_S(\mathfrak{p}_F)/\bar{U}_S(\mathfrak{p}_F^2).$$
\end{proof}

A similar result can be proved for Iwahori-Hecke algebras as well.

\begin{prop}\label{right action of IHA}
    Let $\pi$ be a smooth representation of $G_S$, and suppose $\chi$ is a smooth character of $I_S$. Let $g\in G_S$ be an element normalizing $T_S$. Let $v\in \pi^{(I_S,\chi)}$ (the subspace of elements on which $I_S$ acts by $\chi$, i.e. the $(I_S,\chi)$-isotypic component), and $T_g\in \mathcal{H}(G_S,I_S,\chi)$ be the operator corresponding to the function $\Phi_g\in \mathbb{H}_{\bar{\mathbb{F}}_p}(G_S,I_S,\chi)$ which is supported on $I_SgI_S$ and such that $\Phi_g(g)=1$. Then, we have $$v\,|\,T_g=\sum_{i\in I_S(1)/(I_S(1)\cap (g^{-1}I_S(1)g))}ig^{-1}\cdot v.$$In particular, for $g\in \{w_0,w_0^{-1}\alpha_0^{-1}\}$, we have $$v\,|\,T_{w_0}=\sum_{\lambda\in k_F}\begin{pmatrix}
        1 & [\lambda]\\ 0 & 1
    \end{pmatrix}w_0^{-1}\cdot v\quad \text{and}\quad v\,|\,T_{w_0^{-1}\alpha_0^{-1}}=\sum_{\mu\in k_F}w_0\begin{pmatrix}
        1 & [\mu]\omega_F\\ 0 & 1
    \end{pmatrix}\alpha_0^{-1}\cdot v.$$ 
\end{prop}

\begin{proof}
    Here, we only point out some further reductions. At first, note that $\Phi_g$ is supported on $I_SgI_S=I_SgI_S(1)$. Now, using Proposition 6 of \cite{Barthel} (or by the explicit bijection between the convolution and endomorphism algebras, and Frobenius reciprocity, mentioned in subsection \ref{Generalities}), we have $$f\,|\,T_g=\sum_{g^{\prime}\in I_S\backslash G_S}\Phi_g(g^{\prime})(g^{\prime})^{-1}\cdot f=\sum_{g^{\prime}\in I_S(1)\backslash G_S}\Phi_g(g^{\prime})(g^{\prime})^{-1}\cdot f.$$ Here, the second equality follows from the fact that $\Phi_g(tg^{\prime})(tg^{\prime})^{-1}\cdot f=\Phi_g(g^{\prime})(g^{\prime})^{-1}\cdot f$ for all $t\in T_S(\mathcal{O}_F^{\times})$. Therefore, the above sum is supported on $I_S(1)gI_S(1)$. The remaining proof is same as that of the pro-$p$-Iwahori Hecke algebra case.
\end{proof} 

\begin{defn}
    We define the following maps : $$\mathcal{S}_1:=[v\mapsto \sum_{\lambda\in k_F}\begin{pmatrix}1 & [\lambda]\\ 0 & 1\end{pmatrix}w_0^{-1}\cdot v]: \pi\to \pi$$ and $$\mathcal{S}_2:=[v\mapsto \sum_{\mu\in k_F}w_0\begin{pmatrix}1 & [\mu]\omega_F\\ 0 & 1\end{pmatrix}\alpha_0^{-1}\cdot v]: \pi\to \pi.$$
\end{defn}
We write down the following Lemma whose proof is now clear.
\begin{lm}\label{S_1 and S_2 preserve I_S(1) invarinats}
    Let $\pi$ be a smooth representation of $G_S$. If $v\in \pi^{I_S(1)}$, then $\mathcal{S}_1v,\,\mathcal{S}_2v\in \pi^{I_S(1)}$, and hence we have $\mathcal{S}v=(\mathcal{S}_1\circ \mathcal{S}_2)(v)\in \pi^{I_S(1)}$.
\end{lm}

\subsection{Action of $\mathcal{H}(G_S,I_S,\chi)$ on $(\mathrm{ind}_{K_0}^{G_S}(\sigma))^{(I_S,\chi)}$} At first, we recall some basic results about the $(I_S,\chi)$-isotypic component of $\mathrm{ind}_{K_0}^{G_S}(\sigma)$. Let $\sigma$ be a weight of $K_0$, and $f_n\in (\mathrm{ind}_{K_0}^{G_S}(\sigma))^{I_S(1)}$ denote the function, supported on $K_0\alpha_0^{-n} I_S(1)$ (note that $G_S=\sqcup_{n\in \mathbb{Z}}K_0\alpha_0^{-n} I_S(1)$; see Proposition 3.32 in \cite{Abdellatif}), which is defined as follows : $$f_n(\alpha_0^{-n}):=\begin{cases}
    w_0\cdot v_{\sigma}, & \: n>0\\
    v_{\sigma}, & \: n\leq 0
\end{cases}.$$ Here, $v_{\sigma}\in \sigma^{I_S(1)}$ is the unique vector that generates $\sigma$ (see Carter-Lusztig theory i.e. Theorem \ref{Carter-Lusztig thory}). We recall the following. 

\begin{prop}\label{action of I_S on pro-p invariant basis}
    \begin{enumerate}
        \item The family of functions $(f_n\,|\,n\in \mathbb{Z})$ forms an $\bar{\mathbb{F}}_p$-basis of $(\mathrm{ind}_{K_0}^{G_S}(\sigma))^{I_S(1)}$.

        \item The Iwahori subgroup $I_S$ acts on $f_n$ as follows : 
        $$i\cdot f_n=\begin{cases}
            \chi_{\sigma}(i)f_n, &\: n\leq 0\\
            \chi_{\sigma}^{w_0}(i)f_n, &\: n>0
        \end{cases},$$where $i\in I_S$, and $\chi_{\sigma}$ denotes the character by which $I_S$ acts on $\bar{\mathbb{F}}_p\cdot v_{\sigma}$.
    \end{enumerate}
\end{prop}

\begin{proof}
    See Proposition 3.33 of \cite{Abdellatif}.
\end{proof}

Now, let $\chi$ be a character of $I_S$. Then $\chi|_{I_S(1)}=1$ as $I_S(1)$ is pro-$p$. We can therefore consider $\chi$ as a character of $I_S/I_S(1)\simeq k_F^{\times}$. Consider the isotypic component $(\mathrm{ind}_{K_0}^{G_S}(\sigma))^{(I_S,\chi)}$; let $f$ be a non-zero function in this subspace. Then, $f$ is fixed by $I_S(1)$. So, we write $f=\sum_{n}\lambda_nf_n$, where $\lambda_n$'s are zero for all but finitely many $n$'s. The fact that $I_S$ acts on $f$ by $\chi$ implies, for every $i\in I_S$, we have $$\sum_{n}\lambda_n\chi(i)f_n=\chi(i)f=i\cdot f=\sum_{n}\lambda_n (i\cdot f_n)=\sum_{n}\lambda_n\chi_n(i)f_n,$$ where, by Proposition \ref{action of I_S on pro-p invariant basis}, $\chi_n(i)=\chi_{\sigma}(i)$ if $n\leq 0$ or $\chi(i)=\chi^{w_0}_{\sigma}(i)$ if $n>0$. Since $\lambda_n$ is non-zero for some $n$, and $i\in I_S$ was arbitrary, we have $\chi=\chi_{\sigma}$ or $\chi=\chi_{\sigma}^{w_0}$. We have shown the following.

\begin{prop}\label{ness and suff condition for isotypic comp to be non-zero}
    The isotypic component $(\mathrm{ind}_{K_0}^{G_S}(\sigma))^{(I_S,\chi)}$ is non-zero if and only if $\chi=\chi_{\sigma}$ or $\chi=\chi_{\sigma}^{w_0}$.
\end{prop}

As an immediate consequence of the above computation, we also obtain the following.

\begin{prop}\label{basis of I_S,chi invariants}
    \begin{enumerate}
        \item If $\chi_{\sigma}=\chi_{\sigma}^{w_0}$, then $(\mathrm{ind}_{K_0}^{G_S}(\sigma))^{(I_S,\chi_{\sigma})}=(\mathrm{ind}_{K_0}^{G_S}(\sigma))^{I_S(1)}=\langle f_n\,|\,n\in \mathbb{Z}\rangle$.

        \item If $\chi_{\sigma}\neq \chi_{\sigma}^{w_0}$, then $(\mathrm{ind}_{K_0}^{G_S}(\sigma))^{(I_S,\chi_{\sigma})}=\langle f_n\,|\, n\leq 0\rangle$, and $(\mathrm{ind}_{K_0}^{G_S}(\sigma))^{(I_S,\chi_{\sigma}^{w_0})}=\langle f_n\,|\, n> 0\rangle$.
    \end{enumerate}
\end{prop}

Next, we compute the right action of $\mathcal{H}(G_S,I_S,\chi)$ on $(\mathrm{ind}_{K_0}^{G_S}(\sigma))^{(I_S,\chi)}$. For the isotypic component to be non-zero we take $\chi\in \{\chi_{\sigma},\chi_{\sigma}^{w_0}\}$. 

At first, we consider the situation when $\chi=\chi_{\sigma}=\chi_{\sigma}^{w_0}$.

\begin{prop}\label{action of degenerate IHA on isoty-comp of cpt ind}
    Let $\chi_{\sigma}=\chi_{\sigma}^{w_0}$, so that $(\mathrm{ind}_{K_0}^{G_S}(\sigma))^{(I_S,\chi_{\sigma})}=\langle f_n\,|\,n\in \mathbb{Z}\rangle$. Consider the operators $T_{w_0},\,T_{w_0^{-1}\alpha_0^{-1}}\in \mathcal{H}(G_S,I_S,\chi_{\sigma})$. Then, we have the following : 
    \begin{enumerate}
        \item Let $n\geq 0$, then $$f_{-n}\,|\,T_{w_0}=a_{-n}\cdot f_{-n}\:\text{ and }\:
         f_{-n}\,|\,T_{w_0^{-1}\alpha_0^{-1}}=f_{n+1}.$$

        \item Let $n\geq 1$, then $$f_n\,|\,T_{w_0}=f_{-n}\:\text{ and }\: f_n\,|\,T_{w_0^{-1}\alpha_0^{-1}}=b_{n}\cdot f_n.$$
        Where $a_n$'s and $b_n$'s are scalars.
    \end{enumerate}
\end{prop}

\begin{proof}
    We use Proposition \ref{right action of pro p IHA}. First we compute the functions $f_n\,|\,T_{w_0}$ for $n\geq 1$. Then, $x\in \mathrm{Supp}(f_n\,|\,T_{w_0})\implies xuw_0^{-1}\in  K_0\alpha_0^{-n}I_S(1)$ for some $u\in U_S(\mathcal{O}_F)$. So, the support of $f_n\,|\,T_{w_0}$ is contained in $K_0\alpha_0^{-n}I_S(1)w_0U_S(\mathcal{O}_F)=K_0\alpha_0^{n}I_S(1)$. Hence, we compute $$(f_n\,|\,T_{w_0})(\alpha_0^{n})=\sum_{\lambda\in k_F}f_n(\alpha_0^{n}\begin{pmatrix}
        1 & [\lambda]\\ 0 & 1
    \end{pmatrix}w_0^{-1})=f_n(\alpha_0^{n}w_0^{-1})=w_0^{-1}\cdot f_n(\alpha_0^{-n})=v_{\sigma},$$ where the second equality follows from the fact that $\alpha_0^{n}\begin{pmatrix}
        1 & [\lambda]\\ 0 & 1
    \end{pmatrix}w_0\in K_0\alpha_0^nI_S(1)$ for $\lambda\neq 0$; this is easy to prove by elementary row and column elimination techniques. We conclude that, for $n\geq 1$, we have $$f_n\,|\,T_{w_0}=f_{-n}.$$
    Now, we consider $f_{-n}\,|\,T_{w_0}$ for $n\geq 1$. At first, we note that its support is contained in $$K_0\alpha_0^nI_S(1)w_0U_S(\mathcal{O}_F)=K_0\alpha_0^nU_S(\mathcal{O}_F)w_0U_S(\mathcal{O}_F)\subset K_0\alpha_0^{-n}I_S(1)\cup K_0\alpha_0^nI_S(1),$$where the containment follows by computing the normal form of a matrix of the form $$\begin{pmatrix}\omega_F^{-n} & 0\\ 0 & \omega_F^n\end{pmatrix}\begin{pmatrix}1 & x\\ 0 & 1\end{pmatrix}\begin{pmatrix}0 & -1\\ 1 & 0\end{pmatrix}$$depending of the valuation of $x$ being $1$ or $\geq 1$. So, we compute 
    \begin{equation*}
        \begin{split}
            (f_{-n}\,|\,T_{w_0})(\alpha_0^{-n})&=\sum_{\lambda\in k_F}f_{-n}(\alpha_0^{-n}\begin{pmatrix}1 & [\lambda]\\ 0 & 1\end{pmatrix}w_0^{-1})=\sum_{\lambda\in k_F}\begin{pmatrix}1 & [\lambda]\omega_F^{2n}\\ 0 & 1\end{pmatrix}w_0^{-1}\cdot f_{-n}(\alpha_0^n)\\
            &=\sum_{\lambda\in k_F}w_0^{-1}\begin{pmatrix}1 & 0\\ -[\lambda]\omega_F^{2n} & 1\end{pmatrix}\cdot f_{-n}(\alpha_0^n)=\sum_{\lambda\in k_F}w_0^{-1}\cdot v_{\sigma}=0.
        \end{split}
    \end{equation*}
    Hence, for $n\geq 1$, the function $f_{-n}\,|\,T_{w_0}$ is supported on $K_0\alpha_0^nI_S(1)$, so we have $$f_{-n}\,|\,T_{w_0}=a_{-n}f_{-n}.$$Finally, for $n=0$ it is clear that $f_0\,|\,T_{w_0}$ is supported on $K_0$, and so $$f_0\,|\,T_{w_0}=a_0f_0.$$

    Next, we consider the functions $f_{-n}\,|\, T_{w_0^{-1}\alpha_0^{-1}}$ for $n\geq 0$. Here also, $x\in \mathrm{Supp}(f_{-n}\,|\, T_{w_0^{-1}\alpha_0^{-1}})\implies x\in K_0\alpha_0^nI_S(1)w_0^{-1}\alpha_0^{-1}\bar{U}_S(\mathfrak{p}_F)=K_0\alpha_0^{-(n+1)}I_S(1)$. Therefore, we compute $$(f_{-n}\,|\,T_{w_0^{-1}\alpha_0^{-1}})(\alpha_0^{-(n+1)})=\sum_{\mu\in k_F}f_{-n}(\alpha_0^{-(n+1)}\begin{pmatrix}
        1 & 0\\ [\mu]\omega_F & 1
    \end{pmatrix}\alpha_0w_0)=f_{-n}(\alpha_0^{-n}w_0)=w_0\cdot v_{\sigma},$$where the second equality follows from : $\alpha_0^{-(n+1)}\begin{pmatrix}
        1 & 0\\ [\mu]\omega_F & 1
    \end{pmatrix}\alpha_0w_0\in K_0\alpha_0^{-(n+1)}I_S(1)$ for $\mu\neq 0$; this again is easy to prove using elementary row and column eliminations to reduce the matrix to its normal form. Hence, for $n\geq 0$, we have $$f_{-n}\,|\,T_{w_0^{-1}\alpha_0^{-1}}=f_{n+1}.$$
    Finally, we consider the function $f_n\,|\,T_{w_0^{-1}\alpha_0^{-1}}$ for $n\geq 1$. Note that its support is contained in $$K_0\alpha_0^{-n}I_S(1)w_0^{-1}\alpha_0^{-1}\bar{U}_S(\mathfrak{p}_F)\subset K_0\alpha_0^{n-1}I_S(1)\cup K_0\alpha_0^{-n}I_S(1).$$The above containment can be proved by noting that $K_0\alpha_0^{-n}I_S(1)w_0^{-1}\alpha_0^{-1}\bar{U}_S(\mathfrak{p}_F)=K_0\alpha_0^nU_S(\mathfrak{p}_F)\alpha_0^{-1}\bar{U}_S(\mathfrak{p}_F)$, and computing the normal form of a matrix of the form $$\begin{pmatrix}\omega_F^{-n} & 0\\ 0 & \omega_F^n\end{pmatrix}\begin{pmatrix}1 & x\\ 0 & 1\end{pmatrix}\begin{pmatrix}\omega_F & 0\\ 0 & \omega_F^{-1}\end{pmatrix}$$for $x\in \mathfrak{p}_F$ depending on the valuation of $x$ being $1$ or $\geq 1$. So, we compute 
    \begin{equation*}
        \begin{split}
            (f_n\,|\,T_{w_0^{-1}\alpha_0^{-1}})(\alpha_0^{n-1})&=\sum_{\mu\in k_F}f_n(\alpha_0^{n-1}\begin{pmatrix}1 & 0\\ [\mu]\omega_F & 1\end{pmatrix}\alpha_0w_0)=\sum_{\mu\in k_F}\begin{pmatrix}1 & 0\\ [\mu]\omega_F^{2n-1} & 1\end{pmatrix}\cdot f_n(\alpha_0^nw_0)\\
            &=\sum_{\mu\in k_F}\begin{pmatrix}1 & 0\\ [\mu]\omega_F^{2n-1} & 1\end{pmatrix}w_0\cdot f_n(\alpha_0^{-n})=-\sum_{\mu\in k_F}\begin{pmatrix}1 & 0\\ [\mu]\omega_F^{2n-1} & 1\end{pmatrix}\cdot v_{\sigma}=-\sum_{\mu\in k_F}v_{\sigma}=0.
        \end{split}
    \end{equation*}
    Hence, we conclude that $(f_n\,|\,T_{w_0^{-1}\alpha_0^{-1}})(\alpha_0^{n-1})$ is supported on $K_0\alpha_0^{-n}I_S(1)$, and therefore, for $n\geq 1$, we have $$f_n\,|\,T_{w_0^{-1}\alpha_0^{-1}}=b_nf_n$$where $b_n$'s are scalars. This completes the proof of the Proposition.
\end{proof}

Next, we consider the situation where $\chi_{\sigma}\neq \chi_{\sigma}^{w_0}$. We consider the operators $T_{\alpha_0,\chi_{\sigma}},\,T_{\alpha_0^{-1},\chi_{\sigma}}$ of $\mathcal{H}(G_S,I_S,\chi_{\sigma})$, and the operators $T_{\alpha_0,\chi_{\sigma}^{w_0}},\,T_{\alpha_0^{-1},\chi_{\sigma}^{w_0}}$ of $\mathcal{H}(G_S,I_S,\chi_{\sigma}^{w_0})$.

\begin{prop}\label{action of non-degenerate IHA on isoty-comp of cpt ind}
    Let $\chi_{\sigma}\neq \chi_{\sigma}^{w_0}$. Then : 
    \begin{enumerate}
        \item For $n\geq 0$, we have $f_{-n}\,|\,T_{\alpha_0^{-1},\chi_{\sigma}}=0$ and $f_{-n}\,|\,T_{\alpha_0,\chi_{\sigma}}=f_{-(n+1)}$.

        \item For $n\geq 1$, we have $f_n\,|\,T_{\alpha_0,\chi_{\sigma}^{w_0}}=0$ and $f_n\,|\,T_{\alpha_0^{-1},\chi_{\sigma}^{w_0}}=f_{n+1}$.
    \end{enumerate}
\end{prop}

\begin{proof}
    Recall that by Proposition \ref{right action of IHA}, for $f\in (\mathrm{ind}_{K_0}^{G_S}(\sigma))^{(I_S,\chi)}$ and $T_g\in \mathcal{H}(G_S,I_S,\chi)$, we have $$f\,|\,T_g=\sum_{i\in I_S(1)/(I_S(1)\cap (g^{-1}I_S(1)g))}ig^{-1}\cdot f.$$Now, taking $g\in \{\alpha_0,\alpha_0^{-1}\}$ and observing that the Iwahori decomposition gives $$I_S(1)/(I_S(1)\cap (\alpha_0^{-1}I_S(1)\alpha_0))\simeq U_S(\mathcal{O}_F)/U_S(\mathfrak{p}_F^2)\quad \text{and}\quad I_S(1)/(I_S(1)\cap (\alpha_0I_S(1)\alpha_0^{-1}))\simeq \bar{U}_S(\mathfrak{p}_F)/\bar{U}_S(\mathfrak{p}_F^3),$$we obtain $$f\,|\,T_{\alpha_0}=\sum_{\lambda\in k_F^2}\begin{pmatrix}1 & A(\lambda)\\ 0 & 1\end{pmatrix}\alpha_0^{-1}\cdot f\quad \text{and}\quad f\,|\,T_{\alpha_0^{-1}}=\sum_{\mu\in k_F^2}\begin{pmatrix}1 & 0\\ \omega_FA(\mu) & 1\end{pmatrix}\alpha_0\cdot f.$$
    At first, we consider the first statement. We compute the support of $f_{-n}\,|\,T_{\alpha_0^{-1},\chi_{\sigma}}$; note that the support is contained in $K_0\alpha_0^nI_S(1)\alpha_0^{-1}\bar{U}_S(\mathfrak{p}_F)$. 
    
    When $n=0$, this is contained in the double coset $K_0\alpha_0^{-1}I_S(1)$, hence $$f_0\,|\,T_{\alpha_0^{-1},\chi_{\sigma}}=0,$$ since by Proposition \ref{basis of I_S,chi invariants}, the function $f_0\,|\,T_{\alpha_0^{-1},\chi_{\sigma}}$ is contained in $\langle f_n\,|\,n\leq 0\rangle$ which is supported on $\cup_{k\geq 0}K_0\alpha_0^kI_S(1)$. 

    When $n\geq 1$, we have $$K_0\alpha_0^nI_S(1)\alpha_0^{-1}\bar{U}_S(\mathfrak{p}_F)\subset K_0\alpha_0^{n-1}I_S(1)\cup K_0\alpha_0^{-n}I_S(1)\cup K_0\alpha_0^{-(n+1)}I_S(1).$$This containment follows by observing that for $a\in \mathcal{O}_F$ the matrix $$\begin{pmatrix}\omega_F^{-n} & 0\\ 0 & \omega_F^n\end{pmatrix}\begin{pmatrix}1 & a\\ 0 & 1\end{pmatrix}\begin{pmatrix}\omega_F & 0\\ 0 & \omega_F^{-1}\end{pmatrix}=\begin{pmatrix}\omega_F^{-n+1} & a\omega_F^{-n-1}\\ 0 & \omega_F^{n-1}\end{pmatrix}$$has $\alpha_0^{-(n+1)}$ as its normal form when $a\in \mathcal{O}_F^{\times}$, and $\alpha_0^{-n}$ as its normal form when $a\in \mathfrak{p}_F\setminus \mathfrak{p}_F^2$. If $a\in \mathfrak{p}_F^2$, then the normal form is $\alpha_0^{n-1}$. So, we only need to compute $$(f_{-n}\,|\,T_{\alpha_0^{-1},\chi_{\sigma}})(\alpha_0^{n-1})=\sum_{\mu\in k_F^2}f_{-n}(\alpha_0^{n-1}\begin{pmatrix}1 & 0\\ \omega_FA(\mu) & 1\end{pmatrix}\alpha_0)=\sum_{\mu\in k_F^2}\begin{pmatrix}1 & 0\\ \omega_F^{2n-1}A(\mu) & 1\end{pmatrix}\cdot f_{-n}(\alpha_0^n)=\sum_{\mu\in k_F^2}v_{\sigma}=0.$$Now, the support of $f_{-n}\,|\,T_{\alpha_0,\chi_{\sigma}}$ is contained in $K_0\alpha_0^nI_S(1)\alpha_0U_S(\mathcal{O}_F)= K_0\alpha_0^{n+1}I_S(1)$. So, we compute $$(f_{-n}\,|\,T_{\alpha_0,\chi_{\sigma}})(\alpha_0^{n+1})=\sum_{\lambda\in k_F^2}f_{-n}(\alpha_0^{n+1}\begin{pmatrix}1 & A(\lambda
    )\\ 0 & 1\end{pmatrix}\alpha_0^{-1})=f_{-n}(\alpha_0^n)=v_{\sigma},$$since, for $\lambda\neq (0,0)$ we have $$\alpha_0^{n+1}\begin{pmatrix}1 & A(\lambda
    )\\ 0 & 1\end{pmatrix}\alpha_0^{-1}=\begin{pmatrix}\omega_F^{-n} & \omega_F^{-(n+2)}A(\lambda)\\ 0 & \omega_F^n\end{pmatrix},$$whose normal form is $\alpha_0^{-(n+2)}$ when $A(\lambda)\in \mathcal{O}_F^{\times}$, and $\alpha_0^{-(n+1)}$ when $A(\lambda)\in \mathfrak{p}_F\setminus \mathfrak{p}_F^2$, and that in turn implies that for $\lambda\neq (0,0)$ we have $$\alpha_0^{n+1}\begin{pmatrix}1 & A(\lambda
    )\\ 0 & 1\end{pmatrix}\alpha_0^{-1}\in K_0\alpha_0^{-(n+1)}I_S(1)\cup K_0\alpha_0^{-(n+2)}I_S(1).$$This completes the proof of the first part.

    We now turn to the second part. For $n\geq 1$ we first compute the support of $f_n\,|\,T_{\alpha_0,\chi_{\sigma}^{w_0}}$ which is contained in $K_0\alpha_0^{-n}I_S(1)\alpha_0U_S(\mathcal{O}_F)\subset K_0\alpha_0^{-(n-1)}I_S(1)\cup K_0\alpha_0^nI_S(1)\cup K_0\alpha_0^{n-1}I_S(1)$. However, $f_n\,|\,T_{\alpha_0,\chi_{\sigma}^{w_0}}$ is contained in $\langle f_n\,|\,n\geq 1\rangle$ by Proposition \ref{basis of I_S,chi invariants}, and hence supported on $\cup_{k\geq 1} K_0\alpha_0^{-k}I_S(1)$. So, we compute $$(f_n\,|\,T_{\alpha_0,\chi_{\sigma}^{w_0}})(\alpha_0^{-(n-1)})=\sum_{\lambda\in k_F^2}f_n(\alpha_0^{-n+1}\begin{pmatrix}1 & A(\lambda)\\ 0 & 1\end{pmatrix}\alpha_0^{-1})=\sum_{\lambda\in k_F^2}f_n(\alpha_0^{-n})\sum_{\lambda\in k_F^2}w_0\cdot v_{\sigma}=0.$$Hence, we conclude, for $n\geq 1$, we have $$f_n\,|\,T_{\alpha_0,\chi_{\sigma}^{w_0}}=0.$$Finally, we compute $f_n\,|\,T_{\alpha_0^{-1},\chi_{\sigma}^{w_0}}$. At first, note that its support is contained in $K_0\alpha_0^{-n}I_S(1)\alpha_0^{-1}\bar{U}_S(\mathfrak{p}_F)=K_0\alpha_0^{-(n+1)}I_S(1)$, hence we compute $$(f_n\,|\,T_{\alpha_0^{-1},\chi_{\sigma}^{w_0}})(\alpha_0^{-(n+1)})=\sum_{\mu\in k_F^2}f_n(\alpha_0^{-(n+1)}\begin{pmatrix}1 & 0\\ \omega_FA(\mu) & 1\end{pmatrix}\alpha_0)=f_n(\alpha_0^{-n})=1;$$the second equality follows from the fact that for $\mu\neq (0,0)$ we have $$\alpha_0^{-(n+1)}\begin{pmatrix}1 & 0\\ \omega_FA(\mu) & 1\end{pmatrix}\alpha_0\in K_0\alpha_0^{n}I_S(1)\cup K_0\alpha_0^{n+1}I_S(1),$$which in turn follows by computing the normal form of the above matrix when valuation of $A(\mu)$ is $0$ or $1$. Therefore, we have $$f_n\,|\,T_{\alpha_0^{-1},\chi_{\sigma}^{w_0}}=f_{n+1}$$ for $n\geq 1$. This completes the proof of the second part, and the Proposition.  
\end{proof}

\subsection{Finiteness results} We now prove two results which are analogous to Proposition 18 and 32 of \cite{Barthel}.  The proof of the following codimensionality result is almost similar to that of Corollary 3.3 in \cite{XuHeckeEigenvalues}.

\begin{prop}\label{finite codimension}
    Let $\chi\in \{\chi_{\sigma},\chi_{\sigma}^{w_0}\}$. Then, any non-zero $\mathcal{H}(G_S,I_S,\chi)$-submodule of $(\mathrm{ind}_{K_0}^{G_S}(\sigma))^{(I_S,\chi)}$ has finite codimension.
\end{prop}

\begin{proof}
    Let $M$ be a non-zero $\mathcal{H}(G_S,I_S,\chi)$-submodule of $(\mathrm{ind}_{K_0}^{G_S}(\sigma))^{(I_S,\chi)}$, and take a non-zero $f\in M$.

    At first, we consider the situation where $\chi_{\sigma}\neq \chi_{\sigma}^{w_0}$. Let $\chi=\chi_{\sigma}$. We write $$f=\sum_{0\leq k\leq m}c_kf_{-k}$$for some $0\leq m$, and some $c_k\in \bar{\mathbb{F}}_p$ with $0\leq k\leq m$. We further assume $c_m\neq 0$. We show that the subspace $M^{\prime}$ spanned by $M$ and the set of functions $\{f_0,f_{-1},\dots,f_{-m}\}$ is the whole space $(\mathrm{ind}_{K_0}^{G_S}(\sigma))^{(I_S,\chi)}$. This clearly means $M$ is of finite codimension. Therefore, it suffices to show that $f_{-k}\in M^{\prime}$ for every $k\geq m+1$. Now, applying the operator $T_{\alpha_0,\chi_{\sigma}}$ to $f$ we have, by Proposition \ref{action of non-degenerate IHA on isoty-comp of cpt ind}, that $$f^{\prime}:=f\,|\,T_{\alpha_0,\chi_{\sigma}}=\sum_{0\leq k\leq m}c_kf_{-(k+1)}\in M,$$and hence $f_{-(m+1)}\in M^{\prime}$. Applying $T_{\alpha_0,\chi_{\sigma}}$ repeatedly we obtain $f_{-k}\in M^{\prime}$ for all $k\geq m+1$ as required.

    Next, we let $\chi=\chi_{\sigma}^{w_0}.$ We write $$f=\sum_{1\leq k\leq m}c_kf_k$$as before with some $m\geq 1$ and some $c_k\in \bar{\mathbb{F}}_p$ for $1\leq k\leq m$. Assume $c_m\neq 0$. Here, by applying the operator $T_{\alpha_0^{-1},\chi_{\sigma}^{w_0}}$, we can show similarly that the subspace $M^{\prime}$ spanned by $M$ and $\{f_1,\dots,f_m\}$ is $(\mathrm{ind}_{K_0}^{G_S}(\sigma))^{(I_S,\chi)}$.

    We now turn to the situation where $\chi=\chi_{\sigma}=\chi_{\sigma}^{w_0}$. We write $$f=\sum_{n\leq k\leq m}c_kf_k$$ for some $n,m\in \mathbb{Z}$ and $c_k\in \bar{\mathbb{F}}_p$ for $n\leq k\leq m$. We assume $c_n c_m\neq 0$, and that $n,m>0$. We show that the subspace $M^{\prime}$ spanned by $M$ and $\{f_{-m+1},\dots,f_m\}$ is $(\mathrm{ind}_{K_0}^{G_S}(\sigma))^{(I_S,\chi)}=\mathrm{ind}_{K_0}^{G_S}(\sigma)^{I_S(1)}$. Therefore, we need to show that $f_k\in M^{\prime}$ for $k\geq m+1$ and $k\leq -m$. We consider $f^{\prime}:=f\,|\,T_{w_0}\in M$. By Proposition \ref{action of degenerate IHA on isoty-comp of cpt ind}, we have $$f^{\prime}=f\,|\,T_{w_0}=\sum_{n\leq k\leq m}c_kf_{-k}.$$Since $c_m\neq 0$, we have $f_{-m}\in M^{\prime}$. Again, consider $f^{\prime\prime}:=f^{\prime}\,|\,T_{w_0^{-1}\alpha_0^{-1}}\in M$. By Proposition \ref{action of non-degenerate IHA on isoty-comp of cpt ind}, we have $$f^{\prime\prime}=f^{\prime}\,|\,T_{w_0^{-1}\alpha_0^{-1}}=\sum_{n\leq k\leq m}c_kf_{k+1}.$$ So, we have $f_{m+1}\in M^{\prime}$. We again apply $T_{w_0}$ and then $T_{w_0^{-1}\alpha_0^{-1}}$ repeatedly to obtain $f_k\in M^{\prime}$ for $k\geq m+1$ and $k\leq -m$. On the other hand, if both $n,m\leq 0$, then we apply $T_{w_0^{-1}\alpha_0^{-1}}$ to $f$ to obtain $$f\,|\,T_{w_0^{-1}\alpha_0^{-1}}=\sum_{n\leq k\leq m}c_kf_{-k+1}\in M$$and we are in the previous case.

    Finally, we consider the case when $n\leq 0$ and $m\geq 1$. Here also, we assume $c_nc_m\neq 0$. At first, we show that $f\,|\,T_{w_0}$ and $f\,|\,T_{w_0^{-1}\alpha_0^{-1}}$ cannot both be zero. So, let $f\,|\,T_{w_0^{-1}\alpha_0^{-1}}=0$. This implies, by Proposition \ref{action of degenerate IHA on isoty-comp of cpt ind}, the following : $$\sum_{n\leq k\leq 0}c_kf_{-k+1}+\sum_{1\leq k\leq m}c_kb_kf_k=0.$$Now, since $c_n\neq 0$, this implies $-n+1\leq m$ i.e. $n>-m$. But then applying Proposition \ref{action of degenerate IHA on isoty-comp of cpt ind} again, we get $$f\,|\,T_{w_0}=\sum_{1\leq k\leq m}c_kf_{-k}+\sum_{n\leq k\leq 0}c_ka_kf_{k},$$which is non-zero as $c_m\neq 0$. Therefore, by considering any one of the non-zero functions $f\,|\,T_{w_0}$ or $f\,|\,T_{w_0^{-1}\alpha_0^{-1}}$ we reduce to the case considered before where the basis functions appearing in the presentation of $f$ are all indexed by positive integers or all indexed by non-positive integers. This completes the proof of the Proposition.
\end{proof}

\begin{rmk}
	We mention that in Subsection 3.7.3 of \cite{Abdellatif}, Abdellatif constructs a ''counterexample'' to the above proposition. However, it was mentioned in Remark 3.2 of Kozioł's recent paper\cite{Koziol} that this "counterexample" is incorrect. The author would like to thank Peng Xu for pointing out this remark in Kozioł's paper.
\end{rmk}

\begin{rmk}
	The above codimentionality result is not true in general. For example, it fails for higher rank groups like $\mathrm{GL}_3(F)$; see Theorem 5.1 in \cite{XuHeckeEigenvalues}. In fact, the Proposition 18 of \cite{Barthel} is sensitive to the fact that the irreducible representation of $\mathrm{GL}_2(F)$ under consideration admits central characters. The proof fails to work for irreducibles of $\mathrm{GL}_2(F)$ without central characters; the existence of these has been proved recently by Daniel Le in \cite{le2024irreduciblesmoothrepresentationsdefining}.
\end{rmk}

We now prove the main Theorem of this section, which is the $\mathrm{SL}(2)$-analogue of Proposition 32 of \cite{Barthel}. The proof is almost entirely similar to the one in  \cite{Barthel}, so we just add a few extra lines of details.

\begin{prop}\label{finite dimensionality result}
    Let $\sigma$ be a weight of $K_0$, and let $\pi$ be a smooth irreducible representation of $G_S$. If $\varphi\in \mathrm{Hom}_{G_S}(\mathrm{ind}_{K_0}^{G_S}(\sigma),\pi)$ is a non-zero intertwiner, then the $\mathcal{H}(G_S,K_0,\sigma)$-submodule generated by $\varphi$ is of finite dimension.
\end{prop}

\begin{proof}
    As $\varphi$ is non-zero and $\pi$ is irreducible, $\varphi$ is surjective. But, it cannot be injective as $\mathrm{ind}_{K_0}^{G_S}(\sigma)$ is not irreducible (see Théorème 3.18 in \cite{Abdellatif}). Therefore, since $I_S(1)$ is pro-p, we have $\mathrm{ker}(\varphi)^{I_S(1)}\neq 0$. Now, given any non-zero $v\in \mathrm{ker}(\varphi)^{I_S(1)}$, the $I_S$-subrepresentation $\bar{\mathbb{F}}_p[I_S]\cdot v\subset \mathrm{ker}(\varphi)^{I_S(1)}$ (note that $I_S(1)$ is normal in $I_S$) is finite dimensional, as it can be considered as a representation of $I_S/I_S(1)$. But $I_S/I_S(1)$ is a finite abelian group of order coprime to $p$. So, we can write $\bar{\mathbb{F}}_p[I_S]\cdot v$ as a finite sum of characters. We may replace $v$ by one such non-zero vectors from one of these characters. Hence, we may assume $\mathrm{ker}(\varphi)^{(I_S,\chi)}\neq 0$. By Proposition \ref{ness and suff condition for isotypic comp to be non-zero}, it is forced that such a $\chi\in \{\chi_{\sigma},\chi_{\sigma}^{w_0}\}$. By Frobenius reciprocity, this means $\mathrm{Hom}_{G_S}(\mathrm{ind}_{I_S}^{G_S}(\chi),\mathrm{ker}(\varphi))\neq 0$. We may also assume that $\chi=\chi_{\sigma}$. This is because if $0\neq f\in \mathrm{ker}(\varphi)^{(I_S,\chi_{\sigma}^{w_0})}$, then we may write $f$ in the form $\sum_{1\leq k\leq n}c_kf_k$, where $c_k$ is not zero for some $k$. Now, using Proposition \ref{right action of IHA}, and Proposition \ref{action of degenerate IHA on isoty-comp of cpt ind}, we have $$f^{\prime}:=f\,|\,T_{w_0}=\sum_{\lambda\in k_F}\begin{pmatrix}1 & [\lambda]\\ 0 & 1\end{pmatrix}w_0^{-1}\cdot f=\sum_{1\leq k\leq n}c_kf_{-k}$$ is a non-zero function in $\mathrm{ker}(\varphi)^{(I_S,\chi_{\sigma})}.$

    Now, consider the map $$\varphi^{\ast}:=[T\mapsto \varphi\circ T]:\mathrm{Hom}_{G_S}(\mathrm{ind}_{I_S}^{G_S}(\chi_{\sigma}),\mathrm{ind}_{K_0}^{G_S}(\sigma))\to \mathrm{Hom}_{G_S}(\mathrm{ind}_{I_S}^{G_S}(\chi_{\sigma}),\pi).$$Clearly, $\varphi^{\ast}(\mathrm{Hom}_{G_S}(\mathrm{ind}_{I_S}^{G_S}(\chi_{\sigma}),\mathrm{ker}(\varphi)))=0$. Since, by Proposition \ref{finite codimension}, $\mathrm{Hom}_{G_S}(\mathrm{ind}_{I_S}^{G_S}(\chi_{\sigma}),\mathrm{ker}(\varphi))\simeq \mathrm{ker}(\varphi)^{(I_S,\chi_{\sigma})}$ is a non-zero submodule of $\mathrm{Hom}_{G_S}(\mathrm{ind}_{I_S}^{G_S}(\chi_{\sigma}),\mathrm{ind}_{K_0}^{G_S}(\sigma))\simeq (\mathrm{ind}_{K_0}^{G_S}(\sigma))^{(I_S,\chi_{\sigma})}$, it is of finite codimension. Noting that clearly $\varphi^{\ast}$ is also a $\mathcal{H}(G_S,I_S,\chi_{\sigma})$-module homomorphism, we conclude that the image of $\varphi^{\ast}$ is a finite dimensional submodule of $\mathrm{Hom}_{G_S}(\mathrm{ind}_{I_S}^{G_S}(\chi_{\sigma}),\pi)$. Also, we consider the map $$\Phi^{\ast}:=[T\mapsto \varphi\circ T]:\mathrm{Hom}_{G_S}(\mathrm{ind}_{K_0}^{G_S}(\sigma),\mathrm{ind}_{K_0}^{G_S}(\sigma))\to \mathrm{Hom}_{G_S}(\mathrm{ind}_{K_0}^{G_S}(\sigma),\pi).$$

    Let $\delta_0\in \mathrm{Hom}_{K_0}(\mathrm{ind}_{I_S}^{K_0}(\chi_{\sigma}),\sigma)$ be the intertwiner corresponding (by Frobenius reciprocity) to the intertwiner in $\mathrm{Hom}_{I_S}(\chi_{\sigma},\sigma)$ that carries $1\mapsto v_{\sigma}$. Since, $\delta_0$ is non-zero and $\sigma$ is an irreducible $K_0$ representation, we conclude that $\delta_0$ is surjective. Applying the functor $\mathrm{ind}_{K_0}^{G_S}(-)$ to $\delta_0$ we get $\delta\in \mathrm{Hom}_{G_S}(\mathrm{ind}_{I_S}^{G_S}(\chi_{\sigma}),\mathrm{ind}_{K_0}^{G_S}(\sigma))$. This $\delta$ is also a surjection as $\mathrm{ind}_{K_0}^{G_S}(-)$ is an exact functor (see, for example,
    Exercise 1 in subsection 2.5 of \cite{Bushnell-Henniart}). Now, $\delta$ induces the following maps : 
    $$\delta^{\ast}:=[T\mapsto T\circ \delta]:\mathrm{Hom}_{G_S}(\mathrm{ind}_{K_0}^{G_S}(\sigma),\pi)\to \mathrm{Hom}_{G_S}(\mathrm{ind}_{I_S}^{G_S}(\chi_{\sigma}),\pi),$$and 
    $$\Delta^{\ast}:=[T\mapsto T\circ \delta]:\mathrm{Hom}_{G_S}(\mathrm{ind}_{K_0}^{G_S}(\sigma),\mathrm{ind}_{K_0}^{G_S}(\sigma))\to \mathrm{Hom}_{G_S}(\mathrm{ind}_{I_S}^{G_S}(\chi_{\sigma}),\mathrm{ind}_{K_0}^{G_S}(\sigma)).$$But, as $\delta$ is surjective, both $\delta^{\ast}$ and $\Delta^{\ast}$ are injective. And, clearly from the definitions of the maps $\varphi^{\ast},\,\Phi^{\ast},\,\delta^{\ast},$ and $\Delta^{\ast}$, the following square is commutative : 
    $$\begin{tikzcd}
\mathrm{Hom}_{G_S}(\mathrm{ind}_{K_0}^{G_S}(\sigma),\mathrm{ind}_{K_0}^{G_S}(\sigma)) \arrow{r}{\Phi^{\ast}} \arrow[swap]{d}{\Delta^{\ast}} & \mathrm{Hom}_{G_S}(\mathrm{ind}_{K_0}^{G_S}(\sigma),\pi) \arrow{d}{\delta^{\ast}} \\%
\mathrm{Hom}_{G_S}(\mathrm{ind}_{I_S}^{G_S}(\chi_{\sigma}),\mathrm{ind}_{K_0}^{G_S}(\sigma)) \arrow{r}{\varphi^{\ast}}& \mathrm{Hom}_{G_S}(\mathrm{ind}_{I_S}^{G_S}(\chi_{\sigma}),\pi)
\end{tikzcd}$$Therefore, we conclude that the image $\Phi^{\ast}(\mathrm{End}_{G_S}(\mathrm{ind}_{K_0}^{G_S}(\sigma)))$ of $\Phi^{\ast}$, which is the submodule generated by $\varphi$, is finite dimensional. This completes the proof.
\end{proof}

\section{The case of non-supersingular representations}\label{non-supsing case}

At first, we recall the theory of mod $p$ principal series representations of $G_S$. We start with a smooth $\bar{\mathbb{F}}_p$ character $\eta$ of $B_S$. As the abelianization of $B_S$ is the torus $T_S\cong F^{\times}$, we can consider this $\eta$ as a smooth character of $T_S$ or $F^{\times}$. We then have a short exact sequence of $B_S$ representations : $$0\to V_{\eta}\to \text{Ind}_{B_S}^{G_S}(\eta)\to \eta\to 0,$$ where the map $\text{Ind}_{B_S}^{G_S}(\eta)\to \eta$ is evaluation at the identity matrix, and $V_{\eta}$ is the kernel of this map. We mention the following standard facts : \begin{enumerate}
    \item $V_{\eta}$ is an irreducible representation of $B_S$.
    \item The above short exact sequence splits if and only if $\eta=1,$ the trivial character.
    \item $\text{Ind}_{B_S}^{G_S}(\eta)$ is an irreducible representation of $G_S$ if and only if $\eta\neq 1$. In this case, these representations are called \textit{principal series} of $G_S$.
    \item The $G_S$ representation $\text{Ind}_{B_S}^{G_S}(1)$ is indecomposable and of length $2$, with the trivial representation $1$ as its only subrepresentation and the \textit{Steinberg representation} $\text{St}_{S}:=\frac{\text{Ind}_{B_S}^{G_S}(1)}{1}$ as the only quotient. Hence we have a non-split exact sequence $0\to 1\to \text{Ind}_{B_S}^{G_S}(1)\to \text{St}_S\to 0$ of $G_S$ representations.
    \item As $G_S=B_SI_S(1)\sqcup B_S\beta_0I_S(1)$ (see Lemme 1.7 in \cite{Abdellatif}), the pro-$p$-Iwahori invariants $(\text{Ind}_{B_S}^{G_S}(\eta))^{I_S(1)}$ has dimension $2$ over $\bar{\mathbb{F}}_p$. It is generated by the functions $\ell_{1,\eta}$, supported on $B_SI_S(1)$ and taking value $1$ at $I_2$, and $\ell_{2,\eta}$, supported on $B_S\beta_0I_S(1)$ and taking value $1$ at $\beta_0$. 
\end{enumerate}

The proofs of all these facts can be found in Section $3.4$ of \cite{Abdellatif}.

\begin{lm}\label{only char of sl2 is trivial}
    Let $\pi$ be a smooth representation of $G_S$, and $\bar{\mathbb{F}}_p v$ be a line which is $B_S$-stable. Then $G_S$, and hence $B_S$, acts trivially on this line. Hence, for $\eta\neq 1$ we have $\mathrm{Hom}_{B_S}(\eta,\pi)=0.$
\end{lm}

\begin{proof}
    By smoothness we have some positive integer $m$ such that $\bar{U}_S(\mathfrak{p}_F^m)$ fixes $v$. Consider any arbitrary lower unipotent matrix $\bar{u}(x)\in \bar{U}_S$ for some $x\in F$. There is some $k\in \mathbb{Z}$ such that $y:=x\omega_F^{2k}\in \mathfrak{p}_F^m$. Then $\bar{u}(x)=\alpha_0^k\bar{u}(y)\alpha_0^{-k}$ fixes $v$. The upper and lower unipotent matrices generate $G_S$, and the upper unipotents being the derived subgroup of $B_S$ act trivially on $\bar{\mathbb{F}}_pv$. Therefore, we have that $G_S$ acts trivially on $\bar{\mathbb{F}}_pv$. 
\end{proof}

We have the following immediate corollary, which also follows from the fact that the abelianization of $\mathrm{SL_2}(k)$ for any infinite field $k$ is trivial.

\begin{corl}
    The trivial character is the only smooth $\bar{\mathbb{F}}_p$-character of $G_S$.
\end{corl}

\begin{corl}\label{Hom from line}
    Given a smooth representation $\pi$ of $G_S$, the restriction map induces an isomorphism between the following spaces : $$\mathrm{Hom}_{G_S}(1,\pi)\cong \mathrm{Hom}_{B_S}(1,\pi|_{B_S}).$$
\end{corl}

\begin{proof}
    The existence of a non-zero intertwiner in $\text{Hom}_{B_S}(1,\pi|_{B_S})$ is equivalent to the existence of a $B_S$-stable line in $\pi$. Therefore, the lemma applies and we can extend such an intertwiner trivially to a $G_S$-intertwiner.
\end{proof}

\begin{corl}\label{intertwiner from principal series is injective}
    Let $\pi$ be a smooth representation of $G_S$. Suppose $\eta\neq 1$, and $\varphi\in \mathrm{Hom}_{B_S}(\mathrm{Ind}_{B_S}^{G_S}(\eta),\pi)$ be non-zero. Then $\varphi$ is an injection.
\end{corl}

\begin{proof}
    The proof is same as that of Corollary 5.2 in \cite{PaskunasRestriction}. We will reproduce it to make the article self-contained. So let $0\neq \varphi\in \mathrm{Hom}_{B_S}(\mathrm{Ind}_{B_S}^{G_S}(\eta),\pi)$, and suppose $\mathrm{ker }\varphi\neq 0$. We claim at first that $\mathrm{ker}\varphi=V_{\eta}.$ Otherwise, since $\mathrm{Ind}_{B_S}^{G_S}(\eta)$ as a $B_S$-representation has length $2$, we must have $\mathrm{Ind}_{B_S}^{G_S}(\eta)=\mathrm{ker}\varphi\oplus V_{\eta}.$ This implies $\mathrm{ker}\varphi\cong \eta$ as $B_S$-representations. But this is not possible because $\mathrm{Hom}_{B_S}(\eta,\mathrm{Ind}_{B_S}^{G_S}(\eta))=0$ as $\eta\neq 1$, by Lemma \ref{only char of sl2 is trivial}. But if $\mathrm{ker}\varphi=V_{\eta}$, then $\varphi$ induces a map $\bar{\varphi}\in \mathrm{Hom}_{B_S}(\eta,\pi)$. By Lemma \ref{only char of sl2 is trivial}, $\bar{\varphi}=0$, and hence $\varphi=0$, a contradiction. Hence, $\mathrm{ker}\varphi=0$. 
\end{proof}

\begin{corl}\label{endomorphisms of principal series is one dimensional}
    Let $\eta\neq 1$. Then the restriction map induces as isomorphism $$\mathrm{Hom}_{G_S}(\mathrm{Ind}_{B_S}^{G_S}(\eta),\mathrm{Ind}_{B_S}^{G_S}(\eta))\cong \mathrm{Hom}_{B_S}(\mathrm{Ind}_{B_S}^{G_S}(\eta),\mathrm{Ind}_{B_S}^{G_S}(\eta))$$
\end{corl}

\begin{proof}
    Here also the proof is similar to Corollary 5.3 in \cite{PaskunasRestriction}. We show that $\mathrm{End}_{B_S}(\mathrm{Ind}_{B_S}^{G_S}(\eta))$ is a one dimensional space. For this, let $\varphi_1$ and $\varphi_2$ be two non-zero intertwiners in $\mathrm{End}_{B_S}(\mathrm{Ind}_{B_S}^{G_S}(\eta))$. Then, by the previous Corollary, their restrictions to $V_{\eta}$ gives two non-zero maps in $\mathrm{End}_{B_S}(V_{\eta}).$ Note that $\mathrm{Img}(\varphi_i)=V_{\eta}$, otherwise by the length $2$ condition (as in the proof of the previous Corollary) we will have that $V_{\eta}$ has dimension $1$, which is false.  But $\mathrm{End}_{B_S}(V_{\eta})$ is one dimensional as $V_{\eta}$ is an irreducible $B_S$-representation (see Proposition 2.11 in \cite{Bernstein-Zelevinski}; the proof works over any algebraically closed field). Hence, $\varphi_1|_{V_{\eta}}=\lambda\varphi_2|_{V_{\eta}}$ for some $\lambda\in \bar{\mathbb{F}}_p^{\times}$. But then $\varphi_1-\lambda\varphi_2\in \mathrm{End}_{B_S}(\mathrm{Ind}_{B_S}^{G_S}(\eta))$ is not an injection, and hence must be identically zero.
\end{proof}

We now prove a useful lemma, analogous to  Lemma 5.3 of \cite{xu2024restrictionpmodularrepresentationsu2}, which is, in essence, a rather elegant application of Carter-Lusztig theory.

\begin{lm}\label{generation of non-trivial weight}
    Let $\pi$ be a smooth representation of $G_S$, and $0\neq v\in \pi^{I_S(1)}$. Suppose $I_S$ acts on $v$ by a character. Then, either $\mathcal{S}v=0$, or $\mathcal{S}v$ generates a non-trivial weight of $K_0=\mathrm{SL_2}(\mathcal{O}_F)$.
\end{lm}

\begin{proof}
    Suppose $\mathcal{S}v\neq 0$. We set $v^{\prime}:=\mathcal{S}_2v$ so that $\mathcal{S}v=\mathcal{S}_1v^{\prime}$. By our assumption $v^{\prime}\neq 0$. Now, we recall that $I_S=T_S(\mathcal{O}_F^{\times})I_S(1)$, and since $I_S(1)$ is a pro-$p$-group any character of $I_S$ is trivial on $I_S(1)$. Hence, any character of $I_S$ can be considered as a character of $T_S(\mathcal{O}_F^{\times})$. Therefore, if $I_S$ acts on $v$ by a character $\chi$, and as $I_S(1)$ fixes $v^{\prime}$ by Lemma \ref{S_1 and S_2 preserve I_S(1) invarinats}, $I_S$ will act on $v^{\prime}$ via the character $\chi^{w_0}$, and therefore $I_S$ will act on $\mathcal{S}v$ by $\chi.$
    
    We now consider the $K_0$-representation $\kappa=\bar{\mathbb{F}}_p[K_0]\cdot v^{\prime}$. By Frobenius reciprocity the $I_S$-intertwiner $(1\mapsto v^{\prime}):\chi^{w_0}\to \pi$ corresponds to the $K_0$-intertwiner $\mathrm{ind}_{I_S}^{K_0}(\chi^{w_0})\to \pi$ that takes the generator $\varphi_{\chi^{w_0}}=[1,1]\in \mathrm{ind}_{I_S}^{K_0}(\chi^{w_0})$ to $v^{\prime}$. So we obtain $\kappa$ as a quotient of $\mathrm{ind}_{I_S}^{K_0}(\chi^{w_0})=\mathrm{Ind}_{I_S}^{K_0}(\chi^{w_0})$; this equality is evident from the fact that $K_0$ is compact. This intertwiner maps  $\bar{\mathbb{F}}_p[K_0]\cdot (\mathcal{S}_1\varphi_{\chi^{w_0}})$ onto $\bar{\mathbb{F}}_p[K_0]\cdot (\mathcal{S}v)$. 

    Now, note that any character $\chi$ of $T_S(\mathcal{O}_F^{\times})\cong \mathcal{O}_F^{\times}$ can be considered as a character of $k_F^{\times}$ since $\chi|_{1+\mathfrak{p}_F}=1$, and $\mathcal{O}_F^{\times}/(1+\mathfrak{p}_F)\cong k_F$. Consider the natural map $(f\mapsto \Tilde{f}:=[k\mapsto f(k\,\mathrm{mod}\,\mathfrak{p}_F)]):\mathrm{Ind}_{B}^{\Gamma}(\chi^{w_0})\to \mathrm{Ind}_{I_S}^{K_0}(\chi^{w_0})$. This is a $K_0-$intertwiner, with $K_0$ acting on $\mathrm{Ind}_{B}^{\Gamma}(\chi^{w_0})$ via mod $p$ reduction. This action is obviously smooth as the level one congruence subgroup fixes everything. And, the irreducible representation $\bar{\mathbb{F}}_p[K_0]\cdot (T_{w_0}\varphi_{\chi})$ is mapped onto $\bar{\mathbb{F}}_p[K_0]\cdot (\mathcal{S}_1\varphi_{\chi^{w_0}})$; note the use of equation (\ref{action of T_w_0 on generator}). Therefore, by Carter-Lusztig theory (Theorem \ref{Carter-Lusztig thory}, and Remark \ref{steinberg mod p}), the $K_0$-representation $\bar{\mathbb{F}}_p[K_0]\cdot (\mathcal{S}_1\varphi_{\chi^{w_0}})$, and hence $\bar{\mathbb{F}}_p[K_0]\cdot (\mathcal{S}v)$ is a non-trivial weight of $K_0$. This completes the proof of the lemma.
\end{proof}

\begin{rmk}
    The above lemma is the analogue of Lemma 4.1 of \cite{PaskunasRestriction}. However, in our case we have to deal with the possibility of $\mathcal{S}v=0,$ the analogue of which does not arise in the version of the lemma proved by Paskunas. This additional technical condition (and also the fact that $\mathcal{S}_2v=0$ is also a possibility) makes the proof of one of our main results, Theorem \ref{G_S maps are B_S maps from admissible supersingulars}, considerably more technical than its $\mathrm{GL_2}$ counterpart, which is Theorem 4.4 of \cite{PaskunasRestriction}. This situation also arises in \cite{xu2024restrictionpmodularrepresentationsu2}.
\end{rmk}

The main theorem of this section is the following.

\begin{thm}\label{lifting B_S intertwiners from V_eta to G_S intertwiners from Ind_BS^GS}
    Given a smooth representation $\pi$ of $G_S$, and a smooth character $\eta$ of $B_S$, the restriction map induces an isomorphism between the following spaces of intertwiners : $$\mathrm{Hom}_{G_S}(\mathrm{Ind}_{B_S}^{G_S}(\eta),\pi)\cong \mathrm{Hom}_{B_S}(V_{\eta},\pi|_{B_S}).$$
\end{thm}

\begin{proof}
    We show at first that the restriction map is injective. So, let $\varphi\in \text{Hom}_{G_S}(\text{Ind}_{B_S}^{G_S}(\eta),\pi)$ vanishes on $V_{\eta}$. Then it factors through $\frac{\text{Ind}_{B_S}^{G_S}(\eta)}{V_{\eta}}\cong \eta$ as a $B_S$-intertwiner. So if $\eta\neq 1$, by the Lemma \ref{only char of sl2 is trivial} we have $\varphi=0.$ We therefore assume $\eta=1$; now, if $\varphi\neq 0$, we get a non-zero $B_S$-intertwiner in $\text{Hom}_{B_S}(1,\pi|_{B_S})$, which by Corollary \ref{Hom from line} lifts trivially to a $G_S$ intertwiner in $\text{Hom}_{G_S}(1,\pi)$. Hence, the image of $\varphi$ is a line in $\pi$ on which $G_S$ acts trivially. Consequently, we have $1$ as a quotient of $\text{Ind}_{B_S}^{G_S}(1)$, which is false.

     Recall that the space $(\text{Ind}_{B_S}^{G_S}(\eta))^{I_S(1)}$ is two dimensional, generated by the functions $\ell_{1,\eta}$ and $\ell_{2,\eta}$, supported on $B_SI_S(1)$ and $B_Sw_0I_S(1)$ respectively. Also it can be shown easily that $I_S$ acts on $\ell_{1,\eta}$ by the character $\eta_{+}:=\eta|_{T_S(\mathcal{O}_F^{\times})}$, and on $\ell_{2,\eta}$ by the character $\eta_{-}:=(\eta|_{T_S(\mathcal{O}_F^{\times})})^{w_0}$ (see Lemme 2.10 in \cite{Abdellatif}). Now, by Lemma \ref{S_1 and S_2 preserve I_S(1) invarinats}, we know that $\mathcal{S}\ell_{2,\eta}$ is a linear combination of $\ell_{1,\eta}$ and $\ell_{2,\eta}$. It is easy to see that $\mathcal{S}\ell_{2,\eta}(I_2)=0$. Next, we check that : $$\mathcal{S}\ell_{2,\eta}(\beta_0)=\sum_{\lambda\in k_F^2}\ell_{2,\eta}(\beta_0\begin{pmatrix}1 & A(\lambda)\\0 & 1\end{pmatrix}\alpha_0^{-1})=\eta(\alpha_0),$$ since for $A(\lambda)\neq 0$ we have $$w_0\begin{pmatrix}1 & A(\lambda)\omega_F^{-2}\\ 0 & 1\end{pmatrix}=\begin{pmatrix}A(\lambda)^{-1}\omega_F^2 & -1\\ 0 & A(\lambda)\omega_F^{-2}\end{pmatrix}\begin{pmatrix}1 & 0\\ A(\lambda)^{-1}\omega_F^2 & 1\end{pmatrix}\in B_SI_S(1).$$ In other words we have $\mathcal{S}\ell_{2,\eta}=\eta(\alpha_0)\ell_{2,\eta}.$ Therefore, by Lemma \ref{generation of non-trivial weight}, we have that $\bar{\mathbb{F}}_p[K_0]\cdot \ell_{2,\eta}=\bar{\mathbb{F}}_p[K_0]\cdot (\mathcal{S}\ell_{2,\eta})$ is a non-trivial weight of $K_0$.

     We now show that the restriction map is surjective. So let $\psi\in \mathrm{Hom}_{B_S}(V_{\eta},\pi)$ be non-zero. The map $\ell_{2,\eta}$ is supported on $B_S\beta_0I_S(1)$ and hence it lies in $V_{\eta}$. Therefore, as $V_{\eta}$ is irreducible, $\ell_{2,\eta}$ generates it as an $\bar{\mathbb{F}}_p[B_S]$-module, and hence $\psi(\ell_{2,\eta})\neq 0$. Now, $\psi(\ell_{2,\eta})$ is fixed by $B_S\cap I_S(1)$. Also, $\psi(\mathcal{S}\ell_{2,\eta})=\mathcal{S}\psi(\ell_{2,\eta})=\eta(\alpha_0)\psi(\ell_{2,\eta}),$ that is 

     \begin{equation}\label{S acts on psi(l_2) by scalar}
\psi(\ell_{2,\eta})=\eta(\alpha_0)^{-1}\mathcal{S}\psi(\ell_{2,\eta})
     \end{equation}
     
      Now by smoothness the vector $\psi(\ell_{2,\eta})$ is fixed by $\bar{U}_S(\mathfrak{p}_F^{2k+1})$ for some large $k$. We show that $\bar{U}_S(\mathfrak{p}_F^{2k-1})$ fixes $\psi(\ell_{2,\eta})$ by showing that it fixes $\mathcal{S}\psi(\ell_{2,\eta})$. For this, at first we verify the following claim.

     \textit{Claim : }$U_S(\mathfrak{p}_F^{2k-1})$ fixes $\mathcal{S}_2\psi(\ell_{2,\eta}).$

     \textit{Proof of the Claim : } Take some $b\in \mathfrak{p}_F^{2k-1}$. Then we have :
     \begin{equation*}
     \begin{split}
         &\begin{pmatrix}1 & b\\ 0 & 1\end{pmatrix}\mathcal{S}_2\psi(\ell_{2,\eta})=\sum_{\mu\in k_F}w_0\begin{pmatrix}1 & 0\\ -b & 1\end{pmatrix}\begin{pmatrix}1 & [\mu]\omega_F\\ 0 & 1\end{pmatrix}\alpha_0^{-1}\psi(\ell_{2,\eta})\\
         &=\sum_{\mu\in k_F}w_0\begin{pmatrix}1 & [\mu]\omega_F(1-[\mu]\omega_Fb)^{-1}\\ 0 & 1\end{pmatrix}\begin{pmatrix}(1-[\mu]\omega_Fb)^{-1} & 0\\ 0 & (1-[\mu]\omega_Fb)\end{pmatrix}\begin{pmatrix}1 & 0\\-b(1-[\mu]\omega_Fb)^{-1} & 1\end{pmatrix}\alpha_0^{-1}\psi(\ell_{2,\eta})\\
         &=\sum_{\mu\in k_F}w_0\begin{pmatrix}1 & [\mu]\omega_F(1-[\mu]\omega_Fb)^{-1}\\ 0 & 1\end{pmatrix}\alpha_0^{-1}\psi(\ell_{2,\eta})=\sum_{\mu\in k_F}w_0\begin{pmatrix}1 & [\mu]\omega_F\\ 0 & 1\end{pmatrix}\begin{pmatrix}1 & \ast \omega_F^2\\ 0 & 1\end{pmatrix}\alpha_0^{-1}\psi(\ell_{2.\eta})\\
         &= \sum_{\mu\in k_F}w_0\begin{pmatrix}1 & [\mu]\omega_F\\ 0 & 1\end{pmatrix}\alpha_0^{-1}\psi(\ell_{2,\eta})=\mathcal{S}_2\psi(\ell_{2,\eta}).\text{ ///}
     \end{split}   
     \end{equation*}
     Now, if $b\in \mathfrak{p}_F^{2k-1}$ then we have 
     \begin{equation*}
         \begin{split}
             &\begin{pmatrix}1 & 0\\ b & 1\end{pmatrix}\mathcal{S}\psi(\ell_{2,\eta})=\begin{pmatrix}1 & 0\\ b & 1\end{pmatrix}\mathcal{S}_1(\mathcal{S}_2\psi(\ell_{2,\eta}))\\
             &=\sum_{\lambda\in k_F}\begin{pmatrix}1 & 0\\ b & 1\end{pmatrix}\begin{pmatrix}1 & [\lambda]\\ 0 & 1\end{pmatrix}w_0^{-1}\mathcal{S}_2\psi(\ell_{2,\eta})
             =\sum_{\lambda\in k_F}\begin{pmatrix}1 & [\lambda](1+[\lambda]b)^{-1}\\ 0 & 1\end{pmatrix}w_0^{-1}\mathcal{S}_2\psi(\ell_{2,\eta})\\
             &=\sum_{\lambda\in k_F}\begin{pmatrix}1 & [\lambda]+b_{\lambda}^{\prime}\\ 0 & 1\end{pmatrix}w_0^{-1}\mathcal{S}_2\psi(\ell_{2,\eta})\\
             &=\sum_{\lambda\in k_F}\begin{pmatrix}1 & [\lambda]\\ 0 & 1\end{pmatrix}w_0^{-1}\sum_{\mu\in k_F}w_0\begin{pmatrix}1 & [\mu]\omega_F+b_{\lambda}^{\prime}\\ 0 & 1\end{pmatrix}\alpha_0^{-1}\psi(\ell_{2,\eta})\\
             &=\sum_{\lambda\in k_F}\begin{pmatrix}1 & [\lambda]\\ 0 & 1\end{pmatrix}w_0^{-1}\sum_{\mu\in k_F}w_0\begin{pmatrix}1 & [\mu]\omega_F\\ 0 & 1\end{pmatrix}\begin{pmatrix}1 & \ast\omega_F^2\\ 0 & 1\end{pmatrix}\alpha_0^{-1}\psi(\ell_{2,\eta})=(\mathcal{S}_1\circ\mathcal{S}_2)(\psi(\ell_{2,\eta}))=\mathcal{S}\psi(\ell_{2,\eta}).
         \end{split}
     \end{equation*}
     Here $\ast\omega_F^2$ denotes an element of $\mathfrak{p}_F^2$, and in the fourth equality above $b_{\lambda}^{\prime}$ is a element of $\mathfrak{p}_F^{2k-1}\subset \mathfrak{p}_F$. We apply the above several times to deduce finally that $\bar{U}_S(\mathfrak{p}_F)$ fixes $\mathcal{S}\psi(\ell_{2,\eta})$, and hence $\psi(\ell_{2,\eta})$. By applying the Iwahori decomposition $I_S(1)=U_S(\mathcal{O}_F)\times T_S(1+\mathfrak{p}_F)\times \bar{U}_S(\mathfrak{p}_F)$, we have that $\psi(\ell_{2,\eta})$ is fixed by $I_S(1)$. Now, as $I_S=T_S(\mathcal{O}_F^{\times})I_S(1)$ acts on $\ell_{2,\eta}$ by $\eta_{-}$, therefore $I_S$ acts on $\psi(\ell_{2,\eta})$ by $\eta_{-}$ as well. By appealing to Lemma \ref{generation of non-trivial weight} and Carter-Lusztig theory (Theorem \ref{Carter-Lusztig thory} and Remark \ref{steinberg mod p}), we have $\bar{\mathbb{F}}_p[K_0]\cdot \ell_{2,\eta}\cong\bar{\mathbb{F}}_p[K_0]\cdot (\psi(\ell_{2,\eta}))$, as weights of $K_0$, since the weights are uniquely determined by the character given via the action of $I_S$. Also by Lemma \ref{generation of non-trivial weight}, these weights are non-trivial. 

     Set $\sigma:=\bar{\mathbb{F}}_p[K_0]\cdot (\psi(\ell_{2,\eta}))\subset \pi$; then, by Frobenius reciprocity $\bar{\mathbb{F}}_p[G_S]\cdot (\psi(\ell_{2,\eta}))$ is a quotient of $\mathrm{ind}_{K_0}^{G_S}(\sigma)$, and by equation (\ref{S acts on psi(l_2) by scalar}) it is therefore a quotient of $$\frac{\mathrm{ind}_{K_0}^{G_S}(\sigma)}{(\tau_{\sigma}-\eta(\alpha_0)\cdot \mathrm{Id})(\mathrm{ind}_{K_0}^{G_S}(\sigma))}\cong \mathrm{Ind}_{B_S}^{G_S}(\eta),$$ where the above isomorphism of $G_S-$representations is by Théorème 3.18 of \cite{Abdellatif}. Hence, $\bar{\mathbb{F}}_p[G_S]\cdot (\psi(\ell_{2,\eta}))$ is isomorphic to $\mathrm{Ind}_{B_S}^{G_S}(\eta)$ if $\eta\neq 1$, or $\bar{\mathbb{F}}_p[G_S]\cdot (\psi(\ell_{2,\eta}))$ is isomorphic to $V_{1}$ when $\eta=1$. In both cases, up to scalar multiples, there is only one intertwiner by Corollary \ref{endomorphisms of principal series is one dimensional}; we consider the one whose restriction to $V_{\eta}$ coincides with $\psi$. This proves that the restriction map is surjective.
\end{proof}

The proof of the following corollary is similar to that of Corollary 5.5 in \cite{PaskunasRestriction}.

\begin{corl}\label{intertwiners from principal series}
    Let $\eta\neq 1$ and let $\pi$ be a smooth representation of $G_S$. Then we have $$\mathrm{Hom}_{G_S}(\mathrm{Ind}_{B_S}^{G_S}(\eta),\pi)\cong \mathrm{Hom}_{B_S}(\mathrm{Ind}_{B_S}^{G_S}(\eta),\pi).$$
\end{corl}

\begin{proof}
    Clearly, the restriction map is an injection as $\mathrm{Ind}_{B_S}^{G_S}(\eta)$ is an irreducible $G_S$-representation. Now, let $0\neq \psi\in \mathrm{Hom}_{B_S}(\mathrm{Ind}_{B_S}^{G_S}(\eta),\pi)$. Then by Corollary \ref{intertwiner from principal series is injective}, the following composition of $B_S$-intertwiners is zero  : $\mathrm{Ind}_{B_S}^{G_S}(\eta)\to \pi\to \frac{\pi}{\bar{\mathbb{F}}_p[G_S]\cdot (\psi(V_{\eta}))}.$ Hence, the image of $\psi$ is contained in $\bar{\mathbb{F}}_p[G_S]\cdot (\psi(V_{\eta}))$. This means we can consider $\psi|_{V_{\eta}}$ to be a non-zero element of the space of $B_S$-intertwiners $\mathrm{Hom}_{B_S}(V_{\eta},\bar{\mathbb{F}}_p[G_S]\cdot (\psi(V_{\eta}))).$ By Theorem \ref{lifting B_S intertwiners from V_eta to G_S intertwiners from Ind_BS^GS}, we have that $\bar{\mathbb{F}}_p[G_S]\cdot (\psi(V_{\eta}))$ is a quotient of $\mathrm{Ind}_{B_S}^{G_S}(\eta)$, and as $\mathrm{Ind}_{B_S}^{G_S}(\eta)$ is an irreducible $G_S$-representation, it is in fact isomorphic to $\bar{\mathbb{F}}_p[G_S]\cdot (\psi(V_{\eta}))$. Hence, by Corollary \ref{endomorphisms of principal series is one dimensional}, $\psi$ is $G_S$-linear as well.
\end{proof}

Finally, we note the following corollary which is immediate from Theorem \ref{lifting B_S intertwiners from V_eta to G_S intertwiners from Ind_BS^GS}, by noting that $\mathrm{St}_{S}|_{B_S}\cong V_{1}$ as $B_S$-representations. The later follows from the facts stated in the beginning of this Section.

\begin{corl}\label{intertwiners from steinberg}
    Let $\pi$ be a smooth representation of $G_S$. Then, we have $\mathrm{Hom}_{G_S}(\mathrm{Ind}_{B_S}^{G_S}(1),\pi)\cong \mathrm{Hom}_{B_S}(\mathrm{St}_S,\pi)$.
\end{corl}

\begin{rmk}
    Letting $\pi=\mathrm{Ind}_{B_S}^{G_S}(1)$, we can see that the above result cannot be improved if we replace $\mathrm{St}_S$ with $\mathrm{Ind}_{B_S}^{G_S}(1)$.
\end{rmk}

\section{The case of Supersingular representations}\label{supsing case}

\subsection{Definition of supersingular representations}

Recall that in Subsection \ref{spherical Hecke alg structure} we mentioned the fact that the spherical Hecke algebra $\mathcal{H}(G_S,K_0,\sigma)=\bar{\mathbb{F}}_p[\tau]$ for an operator $\tau$ whose action on the standard functions was given by an explicit formula. We set $\tau_{\sigma}:=\begin{cases} 
      \tau ,& \sigma\neq 1 \\
      \tau+\text{Id}, & \sigma=1. 
   \end{cases}$ We still have $\mathcal{H}(G_S,K_0,\sigma)=\bar{\mathbb{F}}_p[\tau_{\sigma}].$ In a similar manner we can show that the spherical Hecke algebra with respect to the other maximal compact subgroup $K_1$, denoted $\mathcal{H}(G_S,K_1,\sigma^{\alpha})$ is generated as a polynomial algebra in one variable. More precisely, we have $\mathcal{H}(G_S,K_1,\sigma^{\alpha})=\bar{\mathbb{F}}_p[\tau_{\sigma^{\alpha}}^1]$, where the operator $\tau_{\sigma^{\alpha}}^1$ is the analogue of $\tau_{\sigma}$. Here, we use the notation $\alpha$ for the matrix $\mathrm{diag}(1,\omega_F)\in \mathrm{GL}_2(F)$, and $\sigma^{\alpha}$ denotes the twist of $\sigma$ by the inner automorphism given by $\alpha$. We now give the following definition.

\begin{defn}
\begin{enumerate}
    \item A smooth irreducible $\bar{\mathbb{F}}_p[G_S]$-module $\pi$ is said to be \textbf{supersingular with respect to $K_0$} or simply $K_0$-\textbf{supersingular} if it is a quotient of the $\bar{\mathbb{F}}_p[G_S]$-module $\frac{\mathrm{ind}_{K_0}^{G_S}(\sigma)}{\tau_{\sigma}(\mathrm{ind}_{K_0}^{G_S}(\sigma))}.$

    \item A smooth irreducible $\bar{\mathbb{F}}_p[G_S]$-module $\pi$ is said to be \textbf{supersingular with respect to $K_1$} or simply $K_1$-\textbf{supersingular} if it is a quotient of the $\bar{\mathbb{F}}_p[G_S]$-module $\frac{\mathrm{ind}_{K_1}^{G_S}(\sigma^{\alpha})}{\tau_{\sigma^{\alpha}}^1(\mathrm{ind}_{K_1}^{G_S}(\sigma^{\alpha}))}.$
\end{enumerate}
\end{defn}

We have following result which is proved in Proposition $3.20$ of \cite{Abdellatif}.

\begin{prop}\label{parameterization form of supersingulars}
    If $\pi$ is supersingular with respect to $K_0$, and a quotient of the $\bar{\mathbb{F}}_p[G_S]$-module $\frac{\mathrm{ind}_{K_0}^{G_S}(\sigma)}{(\tau_{\sigma}-\lambda\cdot \mathrm{Id})(\mathrm{ind}_{K_0}^{G_S}(\sigma))}$, then $\lambda=0$.
\end{prop}

We also note the following result which follows from using Proposition 3.23 of \cite{Abdellatif}, with the definition above.

\begin{prop}\label{pi is supsing wrt K_0 iff pi^alpha is supsing wrt K_1}
    Let $\pi$ be a smooth irreducible $\bar{\mathbb{F}}_p$-representation of $G_S$. Then, $\pi$ is $K_0$-supersingular if and only if $\pi^{\alpha}$ is $K_1$-supersingular.
\end{prop}

\subsection{Some key Lemmas}

\begin{lm}\label{non-zero intertwiner from c-ind killed by a polynomial of Hecke algebra}
    Let $\pi$ be a $K_0$-supersingular representation of $G_S$, and $\sigma$ a weight of $K_0$. Suppose $\varphi$ be a non-zero $G_S-$intertwiner from $\mathrm{ind}_{K_0}^{G_S}(\sigma)$ to $\pi$. Then, for large enough $k\geq 1$, we have $\varphi\circ \tau_{\sigma}^{k}=0$.
\end{lm}

\begin{proof}
    By Proposition \ref{finite dimensionality result}, the right $\mathcal{H}:=\mathcal{H}(G_S,K_0,\sigma)$-submodule $\varphi\circ \mathcal{H}$ generated by $\varphi$ is finite dimensional. So, if we take the image $\bar{\tau}_{\sigma}(:=n\mapsto n\circ \tau_{\sigma}:\varphi\circ \mathcal{H}\to \varphi\circ \mathcal{H}$) of $\tau_{\sigma}$ in $\mathrm{End}_{\bar{\mathbb{F}}_p}(\varphi\circ \mathcal{H})$, then for the minimal polynomial $m(X)$ of $\bar{\tau}_{\sigma}$, we have $\varphi\circ m(\tau_{\sigma})=0$.  

    Now, suppose $m(X)$ is the polynomial with minimal degree such that $\varphi\circ m(\tau_{\sigma})=0$. Let $\lambda$ be any root of $m(X)$ in $\bar{\mathbb{F}}_p$, and let $m(X)=(X-\lambda)m^{\prime}(X)$. Then $\varphi^{\prime}:=\varphi\circ m^{\prime}(\tau_{\sigma}):\mathrm{ind}_{K_0}^{G_S}(\sigma) \twoheadrightarrow \pi$, which induces a surjection $\frac{\mathrm{ind}_{K_0}^{G_S}(\sigma)}{(\tau_{\sigma}-\lambda \cdot \mathrm{id})(\mathrm{ind}_{K_0}^{G_S}(\sigma))}\twoheadrightarrow \pi$. By Proposition \ref{parameterization form of supersingulars}, we must have $\lambda=0$. Hence, $m(X)=X^k$ for some $k\geq 1$.
\end{proof}

\begin{rmk}
    The proof of the above lemma would have been simpler if $\pi$ was also \textbf{admissible} (meaning $\pi^K$ is finite dimensional for every compact open subgroup $K$ of $G_S$). In such case, it suffices to show that the vector space $\mathrm{Hom}_{G_S}(\mathrm{ind}_{K_0}^{G_S}(\sigma),\pi)\cong \mathrm{Hom}_{K_0}(\sigma,\pi)$ is finite dimensional. For this, one simply shows that $\mathrm{Hom}_{I_S(1)}(\sigma,\pi)$ is finite dimensional. This can be proved by induction on the dimensions of finite dimensional smooth representations of $I_S(1)$. To see this, at first we note that $\sigma$ has a unique line fixed by $I_S(1)$, so we have the short exact sequence $$0\to 1_{I_S(1)}\to \sigma\to \sigma/1_{I_S(1)}\to 0$$ of $I_S(1)$-representations. Applying, the $\mathrm{Hom}_{I_S(1)}(-,\pi)$ functor we get $$0\to \mathrm{Hom}_{I_S(1)}(\sigma/1_{I_S(1)},\pi)\to \mathrm{Hom}_{I_S(1)}(\sigma,\pi)\to \mathrm{Hom}_{I_S(1)}(1_{I_S(1)},\pi)\cong \pi^{I_S(1)}.$$ Since $\pi^{I_S(1)}$ is finite dimensional by admissibility, and $\mathrm{Hom}_{I_S(1)}(\sigma/1_{I_S(1)},\pi)$ is finite dimensional by induction hypothesis, we have that $\mathrm{Hom}_{I_S(1)}(\sigma,\pi)$, and hence $\mathrm{Hom}_{K_0}(\sigma,\pi)$ is finite dimensional as required.
\end{rmk}

We now prove an immediate corollary of the above Lemma, which is analogous to Corollary 3.3 of \cite{PaskunasRestriction}. The proof is inspired by Lemme 2.9 of \cite{Hu}.

\begin{corl}\label{I_S(1) invariants are killed by large powers of S}
    Let $\pi$ be a $K_0$-supersingular representation of $G_S$, and $0\neq v\in \pi^{I_S(1)}$. Then, $\mathcal{S}^iv\in \pi^{I_S(1)}$ for every $i\geq 1$, and for some large enough $k\geq 1$, we have $\mathcal{S}^kv=0$.
\end{corl}

\begin{proof}
    By Lemma \ref{S_1 and S_2 preserve I_S(1) invarinats} it is clear that $\mathcal{S}^iv\in \pi^{I_S(1)}$ for every $i\geq 1$.
    
    At first, we assume that $I_S$ acts on $v$ by a character $\chi.$ If $\mathcal{S}v=0$ we are done. So assume $\mathcal{S}v\neq 0$. Then, $\sigma:=\bar{\mathbb{F}}_p[K_0]\cdot (\mathcal{S}v)$ is a non-trivial weight of $K_0$ in $\pi$ by Lemma \ref{generation of non-trivial weight}. Then by Frobenius reciprocity, we get a surjection $\varphi:\mathrm{ind}_{K_0}^{G_S}(\sigma)\twoheadrightarrow \pi$ that sends $[1,\mathcal{S}v]\mapsto \mathcal{S}v$. As $\sigma$ is non-trivial the action given by $\tau_{\sigma}$ and $\mathcal{S}$ is same (see equation (\ref{action of tau for sigma non-trivial})), so by Lemma \ref{non-zero intertwiner from c-ind killed by a polynomial of Hecke algebra}, we have $0=\varphi\circ \tau_{\sigma}^{k}([1,\mathcal{S}v])=\mathcal{S}^{k+1}v$ for some $k\geq 1$.

    Now, for the general case note that $I_S/I_{S}(1)$ is a finite Abelian group of order coprime to $p$. So by Maschke's theorem, for any $0\neq v\in \pi^{I_S(1)}$, the $I_S$-subrepresentation $\bar{\mathbb{F}}_p[I_S]\cdot v\subset \pi^{I_S(1)}$ is finite dimensional, and hence a sum of characters. So we write $v=\sum_{i} v_i$ such that $I_S$ acts on each $v_i$ by some character. Then, by the argument in the previous paragraph, for each $v_i$ we have a $k_i$ such that $\mathcal{S}^{k_i}v_i=0$. We choose $k$ to be the maximum of all the $k_i$'s. This completes the proof.
\end{proof}
The next Lemma is the analogue of Lemma 3.4 in \cite{PaskunasRestriction}, and the proof is essentially identical.

\begin{lm}\label{for Sv=0 action on v is controlled by B_S}
    Let $\pi$ be a smooth representation of $G_S$ and let $v\in \pi^{I_S(1)}$ such that $\mathcal{S}v=0$. Then we have $$w_0v=-\sum_{\lambda\in k_F^2\setminus \{0,0\}}\begin{pmatrix}\omega_F^{2}A(\lambda)^{-1} & -1\\ 0 & \omega_F^{-2}A(\lambda)\end{pmatrix}v\in \bar{\mathbb{F}}_p[B_S]\cdot v.$$
\end{lm}
\begin{proof}
    We have $$v=-\sum_{\lambda\in k_F^2\setminus \{0,0\}}\alpha_0\begin{pmatrix}1 & A(\lambda)\\ 0 & 1\end{pmatrix}\alpha_0^{-1}v=-\sum_{\lambda\in k_F^2\setminus \{0,0\}}\begin{pmatrix}1 & A(\lambda)\omega_F^{-2}\\ 0 & 1\end{pmatrix}v.$$ Then, it follows that 
    \begin{equation*}
        \begin{split}
            w_0v&=-\sum_{\lambda\in k_F^2\setminus \{0,0\}}\begin{pmatrix}0 & -1\\ 1 & 0\end{pmatrix}\begin{pmatrix}1 & A(\lambda)\omega_F^{-2}\\ 0 & 1\end{pmatrix}v\\
            &=-\sum_{\lambda\in k_F^2\setminus \{0,0\}}\begin{pmatrix}\omega_F^{2}A(\lambda)^{-1} & -1\\ 0 & \omega_F^{-2}A(\lambda)\end{pmatrix}\begin{pmatrix}1 & 0\\ \omega_F^{2}A(\lambda)^{-1} & 1\end{pmatrix}v\\
            &=-\sum_{\lambda\in k_F^2\setminus \{0,0\}}\begin{pmatrix}\omega_F^{2}A(\lambda)^{-1} & -1\\ 0 & \omega_F^{-2}A(\lambda)\end{pmatrix}v\in \bar{\mathbb{F}}_p[B_S]\cdot v.
        \end{split}
    \end{equation*}
    The last equality follows from the fact that for $A(\lambda)\neq 0$ we have $\omega_F^2A(\lambda)^{-1}\in \mathfrak{p}_F$, as the valuation of $A(\lambda)$ is either $0$ or $1$.    
\end{proof}

We now prove a sufficient condition under which a smooth irreducible representation of $G_S$ when restricted to $B_S$ remains irreducible.

\begin{prop}\label{B_S module generated by a non-zero vector has a I_S(1) invariant}
    Let $\pi$ be a smooth irreducible representation of $G_S$, and let $0\neq w\in \pi$. Then, we have : $\pi^{I_S(1)}\cap (\bar{\mathbb{F}}_p[B_S]\cdot w)\neq 0$. 
\end{prop}
\begin{proof}
    As $\pi$ is smooth, there exists $k\geq 0$ such that $w$ is fixed by $\bar{U}_S(\mathfrak{p}_F^{2k+1})$. Then $w^{\prime}:=\alpha_0^{-k}w$ is fixed by $\bar{U}_S(\mathfrak{p}_F)=\alpha_0^{-k}\bar{U}_S(\mathfrak{p}_F^{2k+1})\alpha_0^{k}$. Now, the Iwahori decomposition can be written as $I_S(1)=U(\mathcal{O}_F)\times T_S(1+\mathfrak{p}_F)\times \bar{U}_S(\mathfrak{p}_F)=(I_S(1)\cap B_S)\times \bar{U}_S(\mathfrak{p}_F)$. So, $\bar{\mathbb{F}}_p[I_S(1)]\cdot w^{\prime}=\bar{\mathbb{F}}_p[I_S(1)\cap B_S]\cdot w^{\prime}$ has a non-zero vector fixed by $I_S(1)$ (a pro-$p$ group). But obviously $\bar{\mathbb{F}}_p[I_S(1)\cap B_S]\cdot w^{\prime}\subset \bar{\mathbb{F}}_p[B_S]\cdot w$, and hence $\pi^{I_S(1)}\cap (\bar{\mathbb{F}}_p[B_S]\cdot w)\neq 0$.
\end{proof}

\begin{prop}\label{criterion for restriction to B_S to be irreducible}
     Let $\pi$ be a smooth irreducible representation of $G_S$. Suppose that for any $0\neq w\in \pi$ there exists a non-zero $v\in \pi^{I_S(1)}\cap (\bar{\mathbb{F}}_p[B_S]\cdot w)$ such that $\mathcal{S}v=0$. Then, $\pi|_{B_S}$ is an irreducible $B_S$-representation.
\end{prop}
\begin{proof}
    By Lemma \ref{for Sv=0 action on v is controlled by B_S} we have $w_0v\in \bar{\mathbb{F}}_p[B_S]\cdot v$. Since $G_S=B_SI_S(1)\sqcup B_Sw_0I_S(1)$ (see Lemme 1.7 in \cite{Abdellatif}), and $\pi$ is irreducible, we have $\pi=\bar{\mathbb{F}}_p[G_S]\cdot v=\bar{\mathbb{F}}_p[B_S]\cdot v\subset \bar{\mathbb{F}}_p[B_S]\cdot w$. In conclusion, we have shown that for every $0\neq w\in \pi$ we have $\pi=\bar{\mathbb{F}}_p[B_S]\cdot w$, hence $\pi|_{B_S}$ is an irreducible $B_S-$representation.
\end{proof}

\subsection{Proof of Theorem \ref{main thm 4}}

\begin{thm}\label{restriction of admissible supersingulars to borel is irreducible}
Let $K\in \{K_0,K_1\}$, and let $\pi$ be a $K$-supersingular representation of $G_S$. Then $\pi|_{B_S}$ is irreducible.
\end{thm}
\begin{proof}
    We only prove this for $K=K_0$, owing to Proposition \ref{pi is supsing wrt K_0 iff pi^alpha is supsing wrt K_1}. Let $0\neq w\in \pi$. By, Proposition \ref{B_S module generated by a non-zero vector has a I_S(1) invariant} we can pick a non-zero vector $v\in\bar{\mathbb{F}}_p[B_S]\cdot w$ which is $I_S(1)$-invariant. Now, let $k$ be the least positive integer such that $\mathcal{S}^kv=0$; we know such a $k$ exists by Corollary \ref{I_S(1) invariants are killed by large powers of S}. But by Lemma \ref{S_1 and S_2 preserve I_S(1) invarinats}, we have $0\neq v^{\prime}:=\mathcal{S}^{k-1}v\in \pi^{I_S(1)}\cap\bar{\mathbb{F}}_p[B_S]\cdot w$; here note that the transformation $\mathcal{S}$ is defined by an element of $\bar{\mathbb{F}}_p[B_S]$. Therefore, we can apply Proposition \ref{criterion for restriction to B_S to be irreducible}. This completes the proof.
\end{proof}

We finally prove one of the main theorem of this paper. This proof is inspired by that of  Theorem 5.10 of \cite{xu2024restrictionpmodularrepresentationsu2}.

\begin{thm}\label{G_S maps are B_S maps from admissible supersingulars}
    Let $K\in \{K_0,K_1\}$, and let $\pi$ be a $K$-supersingular representation of $G_S$. Suppose $\pi^{\prime}$ is a smooth representation of $G_S$. Then, $\mathrm{Hom}_{G_S}(\pi,\pi^{\prime})\cong \mathrm{Hom}_{B_S}(\pi,\pi^{\prime})$.
\end{thm}
\begin{proof}
    As before, we prove this for $K=K_0$ because of Proposition \ref{pi is supsing wrt K_0 iff pi^alpha is supsing wrt K_1}. The restriction map is obviously injective. We take a non-zero $\varphi\in \mathrm{Hom}_{B_S}(\pi,\pi^{\prime})$. Now, take a non-zero element $v\in \pi^{I_S(1)}$. We know, by the argument used in the proof of Corollary \ref{I_S(1) invariants are killed by large powers of S}, that the $I_S$-representation $\bar{\mathbb{F}}_p[I_S]\cdot v\subset \pi^{I_S(1)}$ is a finite sum of characters. So, by replacing $v$ by some non-zero vector in $\bar{\mathbb{F}}_p[I_S]\cdot v$, on which $I_S$ acts by a character, we may assume further that $I_S$ acts on $v\in \pi^{I_S(1)}$ by a character $\chi$. It follows that $\varphi(v)\neq 0$, as $\varphi$ is injective by Theorem \ref{restriction of admissible supersingulars to borel is irreducible}. So, by smoothness of $\pi^{\prime}$ we have that $\varphi(v)$ is fixed by $\bar{U}_S(\mathfrak{p}_F^{1+2k})$ for some $k\geq 1$. We will show that there exists a non-zero vector $v_1\in \pi^{I_S(1)}$ on which $I_S$ acts by some character and $\varphi(v_1)$($\neq 0$) is fixed by $\bar{U}_S(\mathfrak{p}_F^{2k-1})$.

    At first, we consider the case when $\mathcal{S}v\neq 0$. Then, we take $v_1:=\mathcal{S}v\in \pi^{I_S(1)}$, by Lemma \ref{S_1 and S_2 preserve I_S(1) invarinats}. Then, $I_S$ acts on $v_1$ by the same character $\chi$, by the same argument in the proof of Lemma \ref{generation of non-trivial weight}. Then, using the same computations as in the proof of Theorem \ref{lifting B_S intertwiners from V_eta to G_S intertwiners from Ind_BS^GS}, we have that $\varphi(v_1)=\mathcal{S}\varphi(v)$ is fixed by $\bar{U}_S(\mathfrak{p}_F^{2k-1})$, since $\varphi(v)$ is fixed by $B_S\cap I_S(1)$.

    Next, we consider the case when $\mathcal{S}v=0$ and $\mathcal{S}_2v\neq 0$. In this case we take $v_1:=\mathcal{S}_2v$, which lies in $\pi^{I_S(1)}$ by Lemma \ref{S_1 and S_2 preserve I_S(1) invarinats}. Then, by our assumption $\mathcal{S}_1w=0$. This means $$\sum_{\lambda\in k_F}\begin{pmatrix}1 & [\lambda]\\ 0 & 1\end{pmatrix}w_0^{-1}\mathcal{S}_2v=0,$$ or equivalently, 
    \begin{equation*}
        \begin{split}
            &v_1=\mathcal{S}_2v=-w_0\sum_{\lambda\in k_F^{\times}}\begin{pmatrix}1 & [\lambda]\\ 0 & 1\end{pmatrix}w_0^{-1}\mathcal{S}_2v=-\sum_{\lambda\in k_F^{\times}}\begin{pmatrix}[\lambda]^{-1} & -1\\ 0 & [\lambda]\end{pmatrix}\begin{pmatrix}1 & 0\\ [\lambda]^{-1} & 1\end{pmatrix}\sum_{\mu\in k_F}\begin{pmatrix}1 & [\mu]\omega_F\\ 0 & 1\end{pmatrix}\alpha_0^{-1}v\\
            &=-\sum_{\lambda\in k_F^{\times}}\begin{pmatrix}[\lambda]^{-1} & -1\\ 0 & [\lambda]\end{pmatrix}\sum_{\mu\in k_F}\begin{pmatrix}1 & [\mu]\omega_F(1+[\lambda]^{-1}[\mu]\omega_F)^{-1}\\ 0 & 1\end{pmatrix}\begin{pmatrix}(1+[\lambda]^{-1}[\mu]\omega_F)^{-1} & 0\\ [\lambda]^{-1} & (1+[\lambda]^{-1}[\mu]\omega_F)\end{pmatrix}\alpha_0^{-1}v\\
            &=-\sum_{\lambda\in k_F^{\times}}\begin{pmatrix}[\lambda]^{-1} & -1\\ 0 & [\lambda]\end{pmatrix}\sum_{\mu\in k_F}\begin{pmatrix}1 & [\mu]\omega_F(1+[\lambda]^{-1}[\mu]\omega_F)^{-1}\\ 0 & 1\end{pmatrix}\alpha_0^{-1}v\\
            &=-\sum_{\lambda\in k_F^{\times}}\begin{pmatrix}[\lambda]^{-1} & -1\\ 0 & [\lambda]\end{pmatrix}\sum_{\mu\in k_F}\begin{pmatrix}1 & [\mu]\omega_F\\ 0 & 1\end{pmatrix}\begin{pmatrix}1 & \ast \omega_F^2\\ 0 & 1\end{pmatrix}\alpha_0^{-1}v\\
            &=-\sum_{\lambda\in k_F^{\times}}\begin{pmatrix}[\lambda]^{-1} & -1\\ 0 & [\lambda]\end{pmatrix}\sum_{\mu\in k_F}\begin{pmatrix}1 & [\mu]\omega_F\\ 0 & 1\end{pmatrix}\alpha_0^{-1}v\\
            &=-\sum_{\lambda\in k_F^{\times}}\begin{pmatrix}1 & -[\lambda]^{-1}\\ 0 & 1\end{pmatrix}\begin{pmatrix}[\lambda]^{-1} & 0\\ 0 & [\lambda]\end{pmatrix}\sum_{\mu\in k_F}\begin{pmatrix}1 & [\mu]\omega_F\\ 0 & 1\end{pmatrix}\alpha_0^{-1}v\\
            &=-\sum_{\lambda\in k_F^{\times}}\chi\begin{pmatrix}[\lambda]^{-1} & 0\\ 0 & [\lambda]\end{pmatrix}\begin{pmatrix}1 & -[\lambda]^{-1}\\ 0 & 1\end{pmatrix}\sum_{\mu\in k_F}\begin{pmatrix}1 & [\mu]\omega_F\\ 0 & 1\end{pmatrix}\alpha_0^{-1}v\\
            &=-\sum_{\lambda\in k_F^{\times}}\chi\begin{pmatrix}[\lambda] & 0\\ 0 & [\lambda^{-1}]\end{pmatrix}\begin{pmatrix}1 & -[\lambda]\\ 0 & 1\end{pmatrix}\sum_{\mu\in k_F}\begin{pmatrix}1 & [\mu]\omega_F\\ 0 & 1\end{pmatrix}\alpha_0^{-1}v
        \end{split}
    \end{equation*}
    Here, note that the 5-th and 7-th equalities follow as $v\in \pi^{I_S(1)}$, and the 9-th equality follows from the assumption that $I_S$ acts on $v$ by the character $\chi$. Also, $\ast\omega_F^2$ denotes an element of $\mathfrak{p}_F^2$. Therefore, we have $$\varphi(v_1)=-\sum_{\lambda\in k_F^{\times}}\chi\begin{pmatrix}[\lambda] & 0\\ 0 & [\lambda]^{-1}\end{pmatrix}\begin{pmatrix}1 & -[\lambda]\\ 0 & 1\end{pmatrix}\sum_{\mu\in k_F}\begin{pmatrix}1 & [\mu]\omega_F\\ 0 & 1\end{pmatrix}\alpha_0^{-1}\varphi(v).$$ Now, by the almost similar computation used the proof of the Claim in Theorem \ref{lifting B_S intertwiners from V_eta to G_S intertwiners from Ind_BS^GS}, we see that $\varphi(v_1)$ is fixed by $\bar{U}_S(\mathfrak{p}_F^{2k-1})$, as required. Also, note that $I_S$ acts on $v_1:=\mathcal{S}_2v$ by a character.

    Finally, consider the case when $\mathcal{S}_2v=0$. Consequently, we have 
    $$\sum_{\mu\in k_F}\begin{pmatrix}1 & [\mu]\omega_F\\ 0 & 1\end{pmatrix}\alpha_0^{-1}v=0.$$Therefore, we get 
    \begin{equation*}
        \begin{split}
            w_0v&=-\sum_{\mu\in k_F^{\times}}w_0\begin{pmatrix}1 & [\mu]\omega_F^{-1}\\ 0 & 1\end{pmatrix}v=-\sum_{\mu\in k_F^{\times}}\begin{pmatrix}[\mu]^{-1}\omega_F & -1\\ 0 &[\mu]\omega_F^{-1}\end{pmatrix}\begin{pmatrix}1 & 0\\ [\mu]^{-1}\omega_F & 1\end{pmatrix}v\\
            &=-\sum_{\mu\in k_F^{\times}}\begin{pmatrix}[\mu]^{-1}\omega_F & -1\\ 0 &[\mu]\omega_F^{-1}\end{pmatrix}v=-\sum_{\mu\in k_F^{\times}}\chi\begin{pmatrix}[\mu^{-1}] & 0\\ 0 & [\mu]\end{pmatrix}\begin{pmatrix}1 & -[\mu]^{-1}\omega_F\\ 0 & 1\end{pmatrix}\alpha_0^{-1}v\\
            &=-\sum_{\mu\in k_F^{\times}}\chi\begin{pmatrix}[\mu] & 0\\ 0 & [\mu^{-1}]\end{pmatrix}\begin{pmatrix}1 & -[\mu]\omega_F\\ 0 & 1\end{pmatrix}\alpha_0^{-1}v,
        \end{split}
    \end{equation*}
    whence, $$\varphi(w_0v)=-\sum_{\mu\in k_F^{\times}}\chi\begin{pmatrix}[\mu] & 0\\ 0 & [\mu^{-1}]\end{pmatrix}\begin{pmatrix}1 & -[\mu]\omega_F\\ 0 & 1\end{pmatrix}\alpha_0^{-1}\varphi(v).$$
    Now, for $b\in \mathfrak{p}_F^{2k-1}$, as in the proof of Theorem \ref{lifting B_S intertwiners from V_eta to G_S intertwiners from Ind_BS^GS}, we have 
    \begin{equation*}
        \begin{split}
            \begin{pmatrix}1 & 0\\ b & 1\end{pmatrix}\varphi(w_0v)&=-\sum_{\mu\in k_F^{\times}}\chi\begin{pmatrix}[\mu] & 0\\ 0 & [\mu^{-1}]\end{pmatrix}\begin{pmatrix}1 & -[\mu]\omega_F+\ast\omega_F^2\\ 0 & 1\end{pmatrix}\begin{pmatrix}(1+\ast\omega_F)^{-1} & 0\\ b & 1+\ast\omega_F\end{pmatrix}\alpha_0^{-1}\varphi(v)\\
            &=-\sum_{\mu\in k_F^{\times}}\chi\begin{pmatrix}[\mu] & 0\\ 0 & [\mu^{-1}]\end{pmatrix}\begin{pmatrix}1 & -[\mu]\omega_F+\ast\omega_F^2\\ 0 & 1\end{pmatrix}\alpha_0^{-1}\varphi(v)\\
            &=-\sum_{\mu\in k_F^{\times}}\chi\begin{pmatrix}[\mu] & 0\\ 0 & [\mu^{-1}]\end{pmatrix}\begin{pmatrix}1 & -[\mu]\omega_F\\ 0 & 1\end{pmatrix}\alpha_0^{-1}v=\varphi(w_0v).
        \end{split}
    \end{equation*}
     Then, for any $\lambda\in k_F$, we note that $\begin{pmatrix}1 & [\lambda]\\ 0 & 1\end{pmatrix}\varphi(w_0v)$ is fixed by $\bar{U}_S(\mathfrak{p}_F^{2k-1})$. This follows from the fact that $U_S(\mathfrak{p}_F)$ and $T_S(1+\mathfrak{p}_F)$ fixes $\varphi(w_0v)$, and from the above computation showing $\bar{U}_S(\mathfrak{p}_F^{2k-1})$ fixes $\varphi(w_0v)$. Setting $u_{\lambda}:=\begin{pmatrix}1 & [\lambda]\\ 0 & 1\end{pmatrix}$, we conclude that elements of the set $\{\varphi(u_{\lambda}w_0v)\,|\,\lambda\in k_F\}$ are fixed by $\bar{U}_S(\mathfrak{p}_F^{2k-1})$. 
     
     Now, we consider the situation when $\mathcal{S}_2v=0$ and $\sum_{\lambda\in k_F}u_{\lambda}w_0v\neq 0$. We take $v_1:=\sum_{\lambda\in k_F}u_{\lambda}w_0v$. Then, $v_1$ can be shown to be $I_S(1)$-invariant by using Iwahori decomposition and showing that it is fixed by the subgroups $U_S(\mathcal{O}_F),\,T_S(1+\mathfrak{p}_F),\,\text{and }\bar{U}_S(\mathfrak{p}_F)$. Also, $\varphi(v_1)=\sum_{\lambda\in k_F}\varphi(u_{\lambda}w_0v)$ is fixed by $\bar{U}_S(\mathfrak{p}_F^{2k-1})$. On the other hand, when $\mathcal{S}_2v=0$ and $\sum_{\lambda\in k_F}u_{\lambda}w_0v=0$, we have 
     \begin{equation*}
         \begin{split}
             v&=w_0\sum_{\lambda\in k_F^{\times}}\begin{pmatrix}1 & [\lambda]\\ 0& 1\end{pmatrix}w_0v=\sum_{\lambda\in k_F^{\times}}\begin{pmatrix}[\lambda]^{-1} & -1\\ 0 & [\lambda]\end{pmatrix}\begin{pmatrix}1 & 0\\ [\lambda]^{-1} & 1\end{pmatrix}w_0v\\
             &=\sum_{\lambda\in k_F^{\times}}\begin{pmatrix}[\lambda]^{-1} & -1\\ 0 & [\lambda]\end{pmatrix}w_0v=\sum_{\lambda\in k_F^{\times}}\chi\begin{pmatrix}[\lambda] & 0\\ 0 & [\lambda]^{-1}\end{pmatrix}\begin{pmatrix}1 & -[\lambda^{-1}]\\ 0& 1\end{pmatrix}w_0v\\
             &=\chi(-I_2)\sum_{\lambda\in k_F}\chi\begin{pmatrix}[\lambda^{-1}] & 0\\ 0 & [\lambda]\end{pmatrix}\begin{pmatrix}1 & [\lambda]\\ 0& 1\end{pmatrix}w_0v.
         \end{split}
     \end{equation*}
     In conclusion, we have $w_0v,\, v\in \mathrm{Span}\{u_{\lambda}w_0v\,|\,\lambda\in k_F^{\times}\}.$ Now, the $K_0$-representation $\bar{\mathbb{F}}_p[K_0]\cdot v$ is spanned by the set $\{v,u_{\lambda}w_0v\,|\,\lambda\in k_F\}$; this follows from the fact that by Bruhat decomposition we have $K_0=I_S\sqcup I_Sw_0I_S$, and $I_S$ acts on $v$ by a character and $I_S(1)$ stabilizes $\mathrm{Span}\{u_{\lambda}w_0v\,|\,\lambda\in k_F^{\times}\}$. This last fact follows from the same computations (done multiple times by now), which can be used to show that $\bar{U}_S(\mathfrak{p}_F)$ stabilizes $\mathrm{Span}\{u_{\lambda}w_0v\,|\,\lambda\in k_F^{\times}\}$. Therefore, as $\bar{\mathbb{F}}_p[K_0]\cdot v$ is a finite dimensional representation, we can choose a weight $\sigma$ inside it. We take $v_1\in \sigma$ to be the unique (up to scalar multiple) vector fixed by $I_S(1)$, on which $I_S$ acts by some character. Hence, $\varphi(v_1)\in \mathrm{Span}\{\varphi(u_{\lambda}w_0v)\,|\,\lambda\in k_F^{\times})\}$, and so the arguments of the previous paragraph show that $\varphi(v_1)$ is fixed by $\bar{U}_S(\mathfrak{p}_F^{2k-1})$. Therefore, we have shown that we can always find some non-zero $v_1\in \pi^{I_S(1)}$,  on which $I_S$ acts by some character, and such that $\varphi(v_1)$ is fixed by $\bar{U}_S(\mathfrak{p}_F^{2k-1})$.

     By repeating this process $k$ many times, we eventually find a non-zero $v_k\in \pi^{I_S(1)}$, on which $I_S$ acts by some character, and such that $\varphi(v_k)$ is fixed by $\bar{U}_S(\mathfrak{p}_F)$. Since, $\varphi(v_k)$ is fixed by $B_S\cap I_S(1)$, we conclude that $\varphi(v_k)$ is fixed by $I_S(1)=(B_S\cap I_S(1))\times \bar{U}_S(\mathfrak{p}_F)$.

     We finish the proof by applying Corollary \ref{I_S(1) invariants are killed by large powers of S} to the vector $v_k$. We obtain some $m\geq 1$ such that $\mathcal{S}^mv_k=0$ and $\mathcal{S}^{m-1}v_k\neq 0$. Setting $v^{\prime}:=\mathcal{S}^{m-1}v_k$, we have that $v^{\prime}$ is $I_S(1)$-invariant, and hence $\varphi(v^{\prime})=\mathcal{S}^{m-1}\varphi(v_k)$ is also $I_S(1)$-invariant. Now, by Lemma \ref{for Sv=0 action on v is controlled by B_S} the condition $\mathcal{S}v^{\prime}=0$ implies 
     \begin{equation*}
         \begin{split}
             w_0v^{\prime}=-\sum_{\lambda\in k_F^2\setminus \{0,0\}}\begin{pmatrix}\omega_F^{2}A(\lambda)^{-1} & -1\\ 0 & \omega_F^{-2}A(\lambda)\end{pmatrix}v^{\prime}.
         \end{split}
     \end{equation*}
     Therefore, as $\varphi$ is a $B_S$-intertwiner, we get 
     $$\varphi(w_0v^{\prime})=-\sum_{\lambda\in k_F^2\setminus \{0,0\}}\begin{pmatrix}\omega_F^{2}A(\lambda)^{-1} & -1\\ 0 & \omega_F^{-2}A(\lambda)\end{pmatrix}\varphi(v^{\prime}).$$ But, we also have $\varphi(\mathcal{S}v^{\prime})=\mathcal{S}\varphi(v^{\prime})=0$, whence, by Lemma \ref{for Sv=0 action on v is controlled by B_S}, we have $$w_0\varphi(v^{\prime})=-\sum_{\lambda\in k_F^2\setminus \{0,0\}}\begin{pmatrix}\omega_F^{2}A(\lambda)^{-1} & -1\\ 0 & \omega_F^{-2}A(\lambda)\end{pmatrix}\varphi(v^{\prime}).$$
     Now, since $G_S=B_SI_S(1)\sqcup B_Sw_0I_S(1)$, and also $\bar{\mathbb{F}}_p[G_S]\cdot v^{\prime}=\pi$, we deduce that $\varphi$ is a $G_S-$intertwiner.
\end{proof}

\section*{Acknowledgment}
The author would like to thank Dr. Peng Xu for several fruitful email correspondences, and also for the encouragement to write down the contents of Section \ref{Hecke algebras and eigenvalues}.

\end{document}